\documentclass[10pt]{article}

\usepackage[T1]{fontenc}                  

\usepackage[applemac]{inputenc}     

\usepackage{microtype}                     

\usepackage[a4paper]{geometry}	

\usepackage{emptypage}			

\usepackage{indentfirst}			

\usepackage{booktabs}			

\usepackage{tabularx}			
\usepackage{multirow}

\usepackage{subfig}				

\usepackage{graphicx}			
\graphicspath{{./Figures/}}

\usepackage[export]{adjustbox}

\usepackage{pgfplots}
\pgfplotsset{legend pos=south east, every axis/.append style={font=\small}}

\usepackage{caption}			

\usepackage{listings}			

\usepackage[font=small]{quoting}	

\usepackage{amsmath,amssymb,amsthm,mathrsfs}        

\usepackage[euler]{textgreek}		

\usepackage{mathtools}			

\usepackage{gensymb}

\usepackage[english]{varioref}		

\usepackage{mparhack,fixltx2e,relsize}		

\usepackage[style=numeric-comp,citestyle=ieee,hyperref,backref,natbib,backend=bibtex]{biblatex}
                                          
\bibliography{Manuscript}

\usepackage[affil-it]{authblk}		

\usepackage{eurosym}			

\usepackage{bookmark}			

\usepackage{empheq}			

\usepackage{lineno}				

\newcommand{\bx}{\mathbf{x}}

\newcommand{\bp}{\mathbf{p}}

\newcommand{\pt}{\partial}

\newcommand{\ce}{\mathcal{E}}
\newcommand{\cs}{\mathcal{S}}
\newcommand{\rn}{\mathbb{R}^n}
\newcommand{\rtwo}{\mathbb{R}^2}
\newcommand{\bbn}{\mathbb{N}}

\newcommand{\tdt}{\tilde{t}}

\newcommand{\ve}{\varepsilon}
\newcommand{\te}{\text}
\providecommand{\keywords}[1]{\textit{Keywords ---} #1}

\theoremstyle{definition}

\theoremstyle{plain} 
\newtheorem{thm}{Theorem}

\theoremstyle{remark} 

\newtheorem*{remark}{Remark}

\newcommand\restr[2]{{
  \left.\kern-\nulldelimiterspace 
  #1 
  \right|_{#2} 
  }}

\DeclareMathOperator\erf{erf}

\begin{document}




\title{Inverse heat conduction to model and optimise a geothermal field}

\date{}


\author[1]{Nadaniela Egidi\thanks{nadaniela.egidi@unicam.it}}

\author[1]{Josephin Giacomini\thanks{Corresponding author, josephin.giacomini@unicam.it}}

\author[1]{Pierluigi Maponi\thanks{pierluigi.maponi@unicam.it}}


\affil[1]{School of Science and Technology - Mathematics Division, University of Camerino, Via Madonna delle Carceri 9, 62032 Camerino (MC), Italy.}

\maketitle

\begin{abstract}
The design of heat exchanger fields is a key phase to ensure the long-term sustainability of such renewable energy systems. This task has to be accomplished by modelling the relevant processes in the complex system made up of different exchangers, where the heat transfer must be considered within exchangers and outside exchangers. We propose a mathematical model for the study of the heat conduction into the soil as consequence of the presence of exchangers. Such a problem is formulated and solved with an analytical approach. On the basis of such analytical solution, we propose an optimisation procedure to compute the best position of the exchangers by minimising the adverse effects of neighbouring devices. Some numerical experiments are used to show the effectiveness of the proposed method also by taking into account a reference approximation procedure of the problem based on a finite difference method.
\end{abstract}



\keywords{optimisation; fluid-solid systems; integral equations; inverse source problem; geothermal field design.}





\section{Introduction}

Low-enthalpy geothermal applications provide sustainable, economic and safe solutions for the future energy supply. Their advantages are the very low environmental impact, the uninterrupted production during the summer and the winter, and the relatively simple installation that can be done in almost all geological settings.

The vertical exchangers are the most popular solution to exchange heat with the soil: a single or double U-shaped pipe is placed into a borehole~\cite{singledoubleU,singledoubleUhelical}, and a carrier fluid is used to exchange heat with the ground material, which at suitable depth has constant temperature throughout the year. Their performances depend on a number of factors, such as exchanger depth, thermophysical properties of the ground material, thermal properties of the filling grout, fluid dynamics parameters of the fluid (velocity, viscosity, diffusivity, and so on). So optimising a geothermal exchanger means developing a fluid dynamics model that depends on those parameters and is able to consider their coupled effect~\cite{1exGlobalOurs}. However, a single geothermal exchanger is often not sufficient to satisfy the energy demand of a building, which requires the installation of multiple devices. A geothermal field consists in an array of devices tidily laid out in a given plot of land. For such systems, the complexity in the optimal design of a single device has to be jointly considered with the difficulty in the analysis of inter-device interference for optimising the geothermal field in a long time perspective. The modelling of such a complex system is a fundamental step in the design of a geothermal field, otherwise the energy extraction or injection for years may introduce a growing thermal anomaly in the ground, such as the local cooling or heating in the interior of the field~\cite{interference}. Consequently, the thermal reservoir of the soil will vanish and the whole field will show a drastically reduced performance~\cite{optim1}. These thermal anomalies are partially compensated by the conductive heat flux through the boundaries of the field, but this is not very effective for systems with too much compact exchangers~\cite{interference}. Thus, the design of an efficient geothermal field should avoid as much as possible thermal anomalies in the soil, even for the long-term heat extraction or injection. A key feature of the geothermal field (strictly connected to the soil temperature profile) is the above-mentioned arrangement of the devices, since the neighbouring exchangers may negatively interfere with each other~\cite{interference}. In fact, a distance between devices of about $7$-$8$ m is usually recommended~\cite{distance}. However, such a prescription is rather empirical, since the complexity of the problem makes unreliable a general estimation of this distance for different geological settings and/or different systems~\cite{long-termBhe}. A more precise design process has to take into account several system features as decision variables (such as the position of the exchangers, their geometrical features and the physical parameters of the ground), while its objective function should require the maximum possible produced energy. However, this quantity is very hard to be computed, and such an objective function would make the optimisation problem practically unsolvable. So, it is usually more effective a simplified objective function that can be computed directly from the soil temperature dynamics that is induced by the presence of the exchangers. The soil temperature dynamics is determined by the energy flux between exchangers and soil, which is influenced by a variety of factors~\cite{influenceFactors1,interference,verdoya-chiozzi,groundwater2}: the number and the geometric characteristics of the devices, the arrangement of the devices, the lithotypes composing the stratigraphy of the soil from which depend both the heat diffusion rate and the presence of aquifers or rainwater infiltrations.

The rainwater infiltrations~\cite{spadoniRichards} and aquifer movements can influence the heat propagation in the soil. However, the estimation of soil moisture dynamics is not an easy task, it requires complex geological surveys and the evaluation of rainwater effects in a long-term perspective. Moreover, inaccurate estimations can introduce large bias in the design process with the consequent worsening of the results obtainable by using purely conduction models. Some studies have considered the advective transport under simplified conditions, for instance the well-known moving infinite line source model investigates the heat transfer only on a horizontal plane and neglects the axial effects due to the finite length of the exchanger~\cite{molina-giraldo}. However, the interesting study in~\cite{verdoya-chiozzi} shows that the heat transport by soil moisture dynamics can be neglected if the Darcy velocity is lower than $10^{-7}$ m/s. So the use of the conductive heat transfer model is a common practice to analyse the temperature of the soil around a geothermal exchanger, where groundwater dynamics is usually discarded, both in numerical computations~\cite{numericConduction,long-termBhe} and in analytical studies~\cite{analitycalAxialeffect,analitycalFast}.

Borehole heat exchanger (BHE) fields have been largely explored, since the time of the first low-temperature geothermal applications in 1970s, so several theoretical and practical contributions have been already produced for BHE fields. The researches have analysed nonresidential ground-source heat pump installations~\cite{kavanaugh}, best practices for designing geothermal systems~\cite{kavanaughNew}, and basic principles for the heating and cooling of buildings~\cite{ashraeFundamentals}. In particular, this last paper presents a strategy for the design of borehole fields known as Ashrae method, which is a fast procedure predicting the BHE field depth, based on the infinite cylindrical source solution (ICS)~\cite{carslaw1986} and a penalty term to take into account the thermal interference between adjacent exchangers. The study~\cite{improvedAshrae2015} showed that previous methods fail in their task, and it proposed a new procedure for the accurate calculation of the penalty term that mimics the original Ashrae method. This improved procedure has also been tested by a sensitivity analysis~\cite{improvedAshrae2016} to show its reliability for an extended range of working conditions. Medium/large BHE fields with unbalanced seasonal loads are analysed in~\cite{unbalancedLoads}, where a new design method for BHE fields based on the superposition of the effects of different BHEs is proposed.

The methods described above provide not only models for the heat transfer in presence of exchangers but also strategies for the optimisation of their activity. In fact, once the model for the description of the heat exchange in the geothermal field has been formulated, it can be exploited in an optimisation procedure to support the design process. Usually, mathematical optimisation methods applied to borehole fields aim to strategically define the workloads of the devices~\cite{optim1} and/or to arrange the borehole positions~\cite{optim2,optim3}. More precisely, optimisation techniques have been used in~\cite{optim2} to strategically operate and arrange in space borehole exchangers. On the other hand, an interesting fact is shown in~\cite{optim3} for homogeneous fields: an optimal load assignment yields the same result of an optimal BHE placement, and the combined optimisation approach produces only slightly better results. Thus, between the two control strategies, the placement optimisation has to be preferred to reduce the management effort of the BHE field. In general, the best position for a couple of exchangers in homogenous soils is the one that maximises the distance between them, in such a way that the mutual interference is minimised. When the array is made of more than two devices, the interference phenomenon among them becomes more complex and a regular lattice may be a suboptimal solution for the position of the devices, especially on heterogeneous soils having quite different diffusivities. So, such a geothermal system needs a detailed analysis.

In this paper, we consider a conductive model for the heat propagation into the soil and, on the basis of this model, we propose an optimisation method for the position of the exchangers that minimises the deviation of the soil temperature from the undisturbed temperature profile in a prescribed time interval. This model takes into account the mutual influence between the soil and the exchangers, since they constitute a coupled system. Thus, both the conductive heat transfer into the soil and the convective heat transfer into the exchangers are considered, as well as their joint action for the heat exchange inside a geothermal field. The key element in the proposed approach is the analytical solution of the heat conductive problem for a generic position of the exchangers. The derivation of this solution is based on standard arguments on the fundamental solution of the heat equation, and an inverse problem for estimating an ad-hoc source term in the conductive problem that mimics the effect of the exchangers. More precisely, this inverse problem allows a simple definition of the source term from the knowledge of the positions of the exchangers, providing in this way a powerful tool for the optimal arrangement of the exchangers. We finally perform a numerical experiment with the proposed method and the numerical results show good reliability and reasonable computational times.

In Section~\ref{sec:generic_form}, the heat conductive problem for a geothermal field is formulated. In Section~\ref{sec:numeric_he}, the finite difference method is used to define a reference procedure for the approximate solution of this heat conductive problem. In Section~\ref{sec:analytic_he}, a forced heat conductive problem is introduced and the corresponding analytical solution is provided. In Section~\ref{sec:inverse_prob}, an inverse source problem for the heat equation is defined and solved. In Section~\ref{sec:optimization}, an optimisation procedure for the placement of the borehole exchangers is described. Section~\ref{sec:results} shows the results of a numerical experiment for the comparison of the analytical and numerical approaches, as well as some applications of the optimisation method. Finally, in Section~\ref{sec:conclusions}, some remarks and further developments are provided.

\section{Formulation of the heat conduction problem}\label{sec:generic_form}

We formulate the problem of the heat propagation into a geothermal field as an initial-boundary value problem for the heat equation. We consider a three-dimensional bounded interval $(-A,A)\times(-B,B)\times(0,H)$, with $A,B,H>0$, containing the exchangers $E_i, i=1,2,\ldots,N_E$, where $N_E$ is the number of exchangers placed in the field. The domain
\begin{equation}\label{domainD}
D = \left\{\bx = (x,y,z)^T \in \mathbb{R}^3: \bx \in \left((-A,A)\times(-B,B)\times(0,H) \setminus \bigcup_{i=1}^{N_E} E_{i} \right) \right\},
\end{equation}
gives the space occupied by the soil surrounding the exchangers. All the exchangers have the same dimensions and the same workloads; they are approximated with parallelepipeds having square cross section with side length $L_E$ and depth $H_E$. We consider the following heat conduction problem:
\begin{equation}\label{global_problem}
\begin{cases}\displaystyle
\frac{\pt u}{\pt t} (\bx,t) - \alpha \Delta u(\bx,t) = 0, & \bx \in D, t \in (0,\overline t), \\
u(\bx,0) = g(\bx), & \bx \in D \cup \Gamma_S \cup \Gamma_E, \\
u(\bx,t) = T_S(\bx), & \bx \in \Gamma_S,t \in (0,\overline t), \\
u(\bx,t) = T_E(\bx,t), & \bx \in \Gamma_E,t \in (0,\overline t),
\end{cases}
\end{equation}
where $[0,\overline t]$ is the time domain, $\pt D = \Gamma_S \cup \Gamma_E$ is the boundary of $D$ where $\Gamma_S$ is the boundary of the parallelepiped $[-A,A]\times[-B,B]\times[0,H]$ and
\begin{equation*}
\Gamma_E = \bigcup_{i=1}^{N_E} \pt E_{i}
\end{equation*}
is the set of the boundaries of all the exchangers, $g$ is a function describing the initial soil temperature distribution that corresponds to the temperature profile of the undisturbed soil, $\alpha$ is the soil thermal diffusivity, $T_S$ is the temperature profile of the undisturbed soil before the system activation, which is supposed to remain constant in time, and $T_E$ is the temperature profile on the exchanger wall. We assume these compatibility conditions: $g(\bx)=T_S(\bx)$, $\bx\in\Gamma_S$, and $g(\bx)=T_E(\bx,0)$, $\bx\in\Gamma_E$. The solution $u$ of problem~\eqref{global_problem} describes how the temperature of the soil varies due to the presence of the exchangers. For simplicity, we analyse the heat transfer in the field up to a maximum depth $H>H_E$; moreover, we have used positive $z$ to describe depth values, so the real geothermal field is modelled by a domain that is formally the symmetric with respect to the plane $z=0$. This choice does not interfere with the description of the physical phenomenon.

\section{Numerical solution of the heat problem}\label{sec:numeric_he}

We briefly describe the approximation of the solution of problem~\eqref{global_problem} by a finite difference method. For simplicity, the space discretisation grid is built for the particular case of exchangers in a regular lattice, however different arrangements of the exchangers can be easily considered with the same procedure. Let $N_x,N_y,N_z$ be the partition size along the $x,y$ and $z$ directions, respectively. Let
\begin{equation*}
h_x=\frac{2A}{N_x}, \qquad h_y=\frac{2B}{N_y}, \qquad h_z=\frac{H}{N_z}
\end{equation*}
be the initial discretisation steps, which are adapted to fit the exchangers cross section. To fix idea, we consider the adaptation method on the $x$ axis. The number $Ne_x$ of grid points that are inside an exchanger is computed on the base of step $h_x$, i.e.,
\begin{equation*}
Ne_x=\left[\frac{L_E}{h_x}\right]+1,
\end{equation*}
where $[a]$ is the nearest integer to $a \in \mathbb{R}$. To have exactly $Ne_x$ grid points inside the exchanger, a new discretisation step is defined. Thus, the total number of grid points in the $x$ direction may increase or decrease by few units with respect to the initial choice $N_x$. For simplicity, we again denote the updated number of grid points by $N_x,N_y,N_z$ and the new discretisation steps by $h_x,h_y,h_z$, respectively. The grid points are
\begin{equation*}
\begin{aligned}
x_i & = -A + i h_x, & i & = 0,\dots,N_x, \\
y_j & = -B + j h_y, & j & = 0,\dots,N_y, \\
z_k & = k h_z, & k & = 0,\dots,N_z.
\end{aligned}
\end{equation*}
We denote the set of indices of all the inner points of the grid by $I$, i.e.,
\begin{equation*}
I = \left\{ (i,j,k) \in \mathbb{N} : 0 \leq i \leq N_x, \, 0 \leq j \leq N_y, \, 0 \leq k \leq N_z \te{ and } (x_i,y_j,z_k)^T \in D \right\},
\end{equation*}
the set of the indices of all the boundary points of the grid in $\Gamma_S$ by $B_S$, i.e.,
\begin{equation*}
B_S = \left\{ (i,j,k) \in \mathbb{N} : 0 \leq i \leq N_x, \, 0 \leq j \leq N_y, \, 0 \leq k \leq N_z \te{ and } (x_i,y_j,z_k)^T \in \Gamma_S \right\},
\end{equation*}
and the set of the indices of all the boundary points of the grid in $\Gamma_E$ by $B_E$, i.e.,
\begin{equation*}
B_E = \left\{ (i,j,k) \in \mathbb{N} : 0 \leq i \leq N_x, \, 0 \leq j \leq N_y, \, 0 \leq k \leq N_z \te{ and } (x_i,y_j,z_k)^T \in \Gamma_E \right\}.
\end{equation*}
We note that the position of each exchanger is slightly adjusted along the $x$-$y$ plane with respect to the discretisation grid, so, since the cross section of the exchangers contains an integer number of partitions, the definition of these sets of indices is particularly simple also for irregular lattices of exchangers. Moreover, there is one constraint the grid must satisfy: we prescribe that at least five inner points in $I$ interpose between an exchanger and the nearest one, in order to have a reliable approximation of the laplacian operator. The discretisation of the time interval $[0,\overline t]$ consists in dividing it into $N_t$ partitions with step $\Delta t$, that is
\begin{equation*}
\Delta t = \frac{\overline t}{N_t},
\end{equation*}
and nodes
\begin{equation*}
t_n=n \Delta t, \qquad n=0,1,\dots,N_t.
\end{equation*}

We use this grid to compute the numerical solution $U(\bx,t)$ of problem~\eqref{global_problem}. We denote
\begin{equation*}
U_{i,j,k}^n \approx U(x_i,y_j,z_k,t_n), \quad (i,j,k) \in I \cup B_S \cup B_E, 0 \leq n \leq N_t.
\end{equation*}
Applying the forward finite difference on the time derivative and the centred finite differences on the Laplacian term, we obtain the usual explicit Euler scheme:
\begin{equation}\label{findiff_he}
\begin{split}
& \frac{U_{i,j,k}^{n+1}-U_{i,j,k}^n}{\Delta t} - \alpha \left(\frac{U_{i+1,j,k}^n-2U_{i,j,k}^n+U_{i-1,j,k}^n}{h_x^2} \right. \\
&  \left. + \frac{U_{i,j+1,k}^n-2U_{i,j,k}^n+U_{i,j-1,k}^n}{h_y^2} + \frac{U_{i,j,k+1}^n-2U_{i,j,k}^n+U_{i,j,k-1}^n}{h_z^2} \right) = 0,
\end{split}
\end{equation}
for $(i,j,k) \in I, n=0,\dots,N_t-1$, where we can reorganise the terms to obtain
\begin{equation}\label{findiff_explicit}
\begin{split}
U_{i,j,k}^{n+1} = & \left( 1 - 2 \alpha \Delta t \left( \frac{1}{h_x^2} + \frac{1}{h_y^2} + \frac{1}{h_z^2} \right) \right) U_{i,j,k}^n + \alpha \frac{\Delta t}{h_x^2} \left(U_{i+1,j,k}^n + U_{i-1,j,k}^n \right) \\
& + \alpha \frac{\Delta t}{h_y^2} \left(U_{i,j+1,k}^n + U_{i,j-1,k}^n\right) + \alpha \frac{\Delta t}{h_z^2} \left(U_{i,j,k+1}^n + U_{i,j,k-1}^n\right),
\end{split}
\end{equation}
for all $(i,j,k) \in I, n=0,\dots,N_t-1$. At the initial time, $U_{i,j,k}^0=g_{i,j,k}$ for all $(i,j,k) \in I \cup B_S \cup B_E$, where $g_{i,j,k}=g(x_i,y_j,z_k)$ for all $(i,j,k) \in I \cup B_S \cup B_E$.
Finally, boundary conditions in problem~\eqref{global_problem} are used to obtain the solution $U_{i,j,k}^{n+1}$ for $(i,j,k) \in B_S \cup B_E$. Therefore, we derived the complete solution $U_{i,j,k}^{n}$ for all $(i,j,k) \in I \cup B_S \cup B_E, n=1,\dots,N_t$, for which the Courant-Friedrichs-Lewy stability condition~\cite{hirsch} must be considered for the choice of the discretisation steps.

\section{Analytical solution of the heat problem}\label{sec:analytic_he}

Problem~\eqref{global_problem} and the corresponding numerical scheme~\eqref{findiff_explicit} give a direct formulation of the heat diffusion in a geothermal field. This approximation approach provides also a quite efficient computational tool when the position of the exchangers is known. On the contrary, it is not so efficient to manage a geothermal field optimisation, because in this case problem~\eqref{global_problem} has to be solved several times with different positions of the exchangers. Thus we reformulate the heat diffusion in a geothermal field by a forced heat equation problem, whose solution can be computed analytically. Problem~\eqref{global_problem} and its numerical solution~\eqref{findiff_explicit} are used in Section~\ref{sec:results} as a benchmark for the comparison with the proposed model.

In Section~\ref{subsec:cond_pr_an}, the diffusive problem is formulated by a forced heat equation, where the source term accounts for the contribution of the geothermal exchangers. In Section~\ref{subsec:analitic_sol}, we illustrate the procedure to obtain the analytical solution by means of the heat kernel, which can be easily used with the new formulation of the problem.

\subsection{The analytical formulation}\label{subsec:cond_pr_an}

Let $\Omega=\{\bx=(x,y,z)^T \in \mathbb{R}^3 : x,y \in \mathbb{R}, 0<z<H\}\supset D$ be the domain for the space variables. The forced heat conductive problem is
\begin{equation}\label{diff_problem}
\begin{cases}\displaystyle
\frac{\pt u}{\pt t} (\bx,t) - \alpha \Delta u(\bx,t) = f(\bx,t), & \bx \in \Omega, t>0, \\
u(\bx,0) = g(\bx), & \bx \in \Omega, \\
u(\bx_0,t) = T_0, & \bx_0=(x,y,0)^T,t>0, \\
u(\bx_H,t) = T_H, & \bx_H=(x,y,H)^T,t>0,
\end{cases}
\end{equation}
where $f$ is the source term, $g$ is the initial temperature distribution, $T_0,T_H$ are the temperature values at the ground level and at depth $H$, respectively,, which are supposed constant, $\alpha$ is the soil thermal diffusivity. We assume these compatibility conditions: $g(\bx_0)=T_0$ and $g(\bx_H)=T_H$. The domain $\Omega$ corresponds to the three-dimensional slab of soil where the system of exchangers is placed, together with the exchangers themselves. So, problem~\eqref{diff_problem} is strictly related with problem~\eqref{global_problem} and the action of the exchangers is described by $f$. The solution $u$ of problem~\eqref{diff_problem} formally describes the temperature of both the soil and the exchangers in the respective positions, but its main objective is the modelling of soil temperature among exchangers, since this is the main information used for the optimisation of the geothermal field.

Let $u$ be the solution of problem~\eqref{diff_problem}, we write $u$ as
\begin{equation}\label{sol_diff_general}
u(\bx,t) = v (\bx,t) + w(z),
\end{equation}
where
\begin{equation}\label{translation}
w(z) = \frac{z}{H}(T_H-T_0) +T_0
\end{equation}
is a suitable translation along the $z$ direction such that $v(\bx,t)$ satisfies the homogeneous boundary conditions, i.e., $v(\bx,t)$ is the solution of the problem
\begin{equation}\label{diff_problem_hom}
\begin{cases}\displaystyle
\frac{\pt v}{\pt t} (\bx,t) - \alpha \Delta v(\bx,t) = f(\bx,t), & \bx \in \Omega, t>0, \\
v(\bx,0) = g(\bx)-w(z), & \bx \in \Omega, \\
v(\bx_0,t) = 0, & \bx_0=(x,y,0)^T,t>0, \\
v(\bx_H,t) = 0, & \bx_H=(x,y,H)^T,t>0.
\end{cases}
\end{equation}
In the following, we describe an explicit formula for the solution of problem~\eqref{diff_problem_hom} and in turn of problem~\eqref{diff_problem}, as well as an approximation of the solution of problem~\eqref{global_problem}, where the exchangers prescribe a proper source term $f$ in problem~\eqref{diff_problem_hom}.

\subsection{The solution of the forced heat transfer problem}\label{subsec:analitic_sol}

We illustrate the main steps for obtaining the analytical solution of problem~\eqref{diff_problem_hom}. This procedure uses standard arguments on the theory of partial differential equations, so the reader can find all the details in~\cite{evans}.

The fundamental solution of the heat equation in $\rn$ is
\begin{equation}\label{solution_fund}
\mathcal{G}(\bx,t;\boldsymbol\xi,\tau) =
\begin{cases}\displaystyle
\frac{1}{\left(4\alpha \pi (t-\tau)\right)^{n/2}} e^{-\frac{||\bx-\boldsymbol\xi||^2}{4\alpha (t-\tau)}}, & \bx,\boldsymbol\xi \in \mathbb{R}^n, 0<\tau<t<+\infty, \\
0, & \bx,\boldsymbol\xi \in \mathbb{R}^n, -\infty<\tau<t<0,
\end{cases}
\end{equation}
where $||\cdot||$ is the Euclidean norm; $\mathcal{G}$ represents the temperature at location $\bx$ and time $t$ resulting from an instantaneous point source of heat releasing a unit of thermal energy at location $\boldsymbol\xi$ and time $\tau$. It solves the problem
\begin{equation}\label{onedim_hom_he}
\begin{cases}\displaystyle
\frac{\pt \mathcal{G}}{\pt t}(\bx,t;\boldsymbol\xi,\tau) - \alpha \Delta \mathcal{G}(\bx,t;\boldsymbol\xi,\tau) = 0, & \bx,\boldsymbol\xi \in \rn, 0<\tau<t<+\infty, \\
\mathcal{G}(\bx,0;\boldsymbol\xi,0)=\delta(\bx-\boldsymbol\xi), & \bx,\boldsymbol\xi \in \rn,
\end{cases}
\end{equation}
where $\delta(\bx-\boldsymbol\xi)$ is the Dirac distribution on $\rn$. Function~\eqref{solution_fund} gives a powerful tool for obtaining solutions of forced heat diffusion problems~\cite{heatGreenfun}, even in presence of inhomogeneous initial conditions and bounded domains.

Let us consider the initial-value problem with inhomogeneous heat equation and an inhomogeneous initial condition. The solution is given by the following theorem:
\begin{thm}\label{thm:total_inhom}
Let $f$ be a locally integrable function in $\mathbb{R}^{n+1}$ that is bounded in each time strip $0\leq t\leq T$, for some $T>0$. Let $g$ be a continuous bounded function in the space variables. Then the problem
\begin{equation}\label{complete_problem}
\begin{cases}\displaystyle
\frac{\pt u}{\pt t}(\bx,t) - \alpha \Delta u(\bx,t) = f(\bx,t), & (\bx,t) \in \rn \times (0,+\infty), \\
u(\bx,0) = g(\bx), & \bx \in \rn,
\end{cases}
\end{equation}
has solution
\begin{equation}\label{sol_complete}
u(\bx,t) = \int_{\rn} \mathcal{G}(\bx,t;\boldsymbol\xi,0) g(\boldsymbol\xi) d\boldsymbol\xi + \int_0^t \int_{\rn} \mathcal{G}(\bx,t;\boldsymbol\xi,\tau) f(\boldsymbol\xi,\tau)d\boldsymbol\xi d\tau.
\end{equation}
\end{thm}
\begin{proof}
The proof is based on the superposition principle, being the heat equation a linear partial differential equation. Thus, it is sufficient to add the solution of the associated homogeneous equation, i.e., with $f=0$, satisfying the inhomogeneous initial condition and the solution of problem~\eqref{complete_problem} equipped with initial condition $g=0$, see~\cite{vladimirov} for the complete proof of this theorem.
\end{proof}

\begin{remark}
This result for the unbounded spatial domain must be jointly considered with the case of bounded domain, since problem~\eqref{diff_problem_hom} is defined for $(x,y) \in \rtwo$ and $z\in(0,H)$. Let us consider the homogeneous conductive problem on a rod of length $H>0$, that is
\begin{equation}\label{rod_problem}
\begin{cases}\displaystyle
\frac{\pt u}{\pt t} (x,t) - \alpha \frac{\pt^2 u}{\pt x^2}(x,t) = 0, & x \in (0,H), t>0, \\
u(x,0) = g(x), & x \in (0,H), \\
u(0,t) = u(H,t) = 0, & t>0.
\end{cases}
\end{equation}
The solution of this problem can be computed in a way similar to problem~\eqref{complete_problem}, in fact, it has the following form
\begin{equation}\label{sol_rod}
u(x,t) = \int_{0}^H \mathcal{G}(x,t;\xi,0) g(\xi) d\xi,
\end{equation}
where $\mathcal{G}$ is the fundamental solution of the heat equation in the domain $(0,H)$ with homogeneous Dirichlet boundary conditions. From the separation of variables we have that
\begin{equation}\label{green_rod}
\mathcal{G}(x,t;\xi,\tau) = \frac{2}{H} \sum_{r=1}^{\infty} \sin \left(\frac{r\pi x}{H} \right) \sin \left(\frac{r\pi \xi}{H} \right) e^{-\frac{r^2\pi^2 \alpha (t-\tau)}{H^2}}.
\end{equation}
\end{remark}

In the following theorem we use the previous results to obtain the fundamental solution of the heat equation in the domain $\Omega$ of problem~\eqref{diff_problem_hom}.
\begin{thm}\label{separable_domain}
Let $\Omega = \mathbb{R}^2 \times (0,H)$. Let $\mathcal{G}^1,\mathcal{G}^2,\mathcal{G}^3$ be the fundamental functions of the heat equation in the space variables $x,y,z$, respectively; in particular, $\mathcal{G}^1,\mathcal{G}^2$ are relative to problem~\eqref{complete_problem} in $\mathbb{R}$ and $\mathcal{G}^3$ to problem~\eqref{rod_problem} in $(0,H)$. Then $\mathcal{G}=\mathcal{G}^1 \mathcal{G}^2 \mathcal{G}^3$ is the fundamental function of the heat equation in $\Omega$ satisfying the boundary conditions of problem~\eqref{diff_problem_hom}.
\end{thm}
\begin{proof}
	\begin{equation*}
	\begin{split}\displaystyle
	\frac{\pt \mathcal{G}}{\pt t} - \alpha \Delta \mathcal{G} & = \frac{\pt \mathcal{G}^1}{\pt t} \mathcal{G}^2 \mathcal{G}^3 + \mathcal{G}^1 \frac{\pt \mathcal{G}^2}{\pt t} \mathcal{G}^3 + \mathcal{G}^1 \mathcal{G}^2 \frac{\pt \mathcal{G}^3}{\pt t} - \alpha \frac{\pt^2 \mathcal{G}^1}{\pt x^2} \mathcal{G}^2 \mathcal{G}^3 - \alpha \mathcal{G}^1\frac{\pt^2 \mathcal{G}^2}{\pt y^2} \mathcal{G}^3 - \alpha \mathcal{G}^1 \mathcal{G}^2 \frac{\pt^2 \mathcal{G}^3}{\pt z^2} \\
	& = \left(\frac{\pt \mathcal{G}^1}{\pt t} - \alpha \frac{\pt^2 \mathcal{G}^1}{\pt x^2}\right)\mathcal{G}^2 \mathcal{G}^3 + \left(\frac{\pt \mathcal{G}^2}{\pt t}- \alpha \frac{\pt^2 \mathcal{G}^2}{\pt y^2}\right)\mathcal{G}^1 \mathcal{G}^3 + \left(\frac{\pt \mathcal{G}^3}{\pt t}- \alpha \frac{\pt^2 \mathcal{G}^3}{\pt z^2}\right)\mathcal{G}^1 \mathcal{G}^2 \\
	& = 0,
	\end{split}
	\end{equation*}
since $\mathcal{G}^1,\mathcal{G}^2,\mathcal{G}^3$ verify the homogeneous heat equation. Also for the initial condition and in distributional sense, we have
	\begin{equation*}
	\begin{split}
	\mathcal{G}(\bx,0;\boldsymbol\xi,0) & = \mathcal{G}^1(x,0;\xi,0) \mathcal{G}^2(y,0;\eta,0) \mathcal{G}^3(z,0;\zeta,0) \\
	& = \delta(x-\xi) \delta(y-\eta) \delta(z-\zeta) \\
	& = \delta(\bx-\boldsymbol\xi),
	\end{split}
	\end{equation*}
where $\boldsymbol\xi=(\xi,\eta,\zeta) \in \Omega$. Finally, the boundary condition in $\bx_0$ is satisfied, in fact,
	\begin{equation*}
	\begin{split}
	\mathcal{G}(\bx_0,t;\boldsymbol\xi,\tau) & = \mathcal{G}^1(x,t;\xi,\tau) \mathcal{G}^2(y,t;\eta,\tau) \mathcal{G}^3(0,t;\zeta,\tau) \\
	& = 0, \\
	\end{split}
	\end{equation*}
since $\mathcal{G}^3$ satisfies the null boundary condition in $z=0$, and a similar argument holds for the boundary $\bx_H$.
\end{proof}
In our case, since the domain $\Omega$ of problem~\eqref{diff_problem_hom} is rectangular, $\Omega=\mathbb{R}^2\times(0,H)$, the fundamental solution $\mathcal{G}(\bx,t;\boldsymbol\xi,\tau)$ of the heat equation in $\Omega$ is given by the product of the fundamental solutions related to proper one-dimensional problems. More precisely, from~\eqref{solution_fund} and~\eqref{green_rod}, we obtain
\begin{equation}\label{green_function}
\begin{split}
\displaystyle \mathcal{G}(\bx,t;\boldsymbol\xi,\tau) = &\mathcal{G}^1(x,t;\xi,\tau) \mathcal{G}^2(y,t;\eta,\tau) \mathcal{G}^3(z,t;\zeta,\tau)\\
\displaystyle = & \frac{1}{\sqrt{4 \alpha \pi (t-\tau)}} e^{-\frac{(x-\xi)^2}{4 \alpha (t-\tau)}} \frac{1}{\sqrt{4 \alpha \pi (t-\tau)}} e^{-\frac{(y-\eta)^2}{4 \alpha (t-\tau)}} \\
\displaystyle & \frac{2}{H} \sum_{r=1}^{\infty} \sin \left(\frac{r\pi z}{H} \right) \sin \left(\frac{r\pi \zeta}{H} \right) e^{-\frac{r^2\pi^2 \alpha (t-\tau)}{H^2}} \\
\displaystyle = & \frac{1}{2 \alpha \pi H(t-\tau)}e^{-\frac{(x-\xi)^2+(y-\eta)^2}{4 \alpha (t-\tau)}} \sum_{r=1}^{\infty} \sin \left(\frac{r\pi z}{H} \right) \sin \left(\frac{r\pi \zeta}{H} \right) e^{-\frac{r^2\pi^2 \alpha (t-\tau)}{H^2}}.
\end{split}
\end{equation}
Therefore, from Theorem~\ref{thm:total_inhom} and relation~\eqref{sol_diff_general}, the complete solution of problem~\eqref{diff_problem} can be expressed as follows
\begin{equation}\label{sol_diffusion}
\begin{split}
u(\bx,t) & = \int_{\Omega} \mathcal{G}(\bx,t;\boldsymbol\xi,0) \left(g(\boldsymbol\xi)-\frac{\zeta}{H}(T_H-T_0)-T_0\right) d\boldsymbol\xi \\
& \hspace{10pt} + \int_0^t \int_{\Omega} \mathcal{G}(\bx,t;\boldsymbol\xi,\tau)f(\boldsymbol\xi,\tau)d\boldsymbol\xi d\tau + \frac{z}{H}(T_H-T_0) +T_0,
\end{split}
\end{equation}
where $\mathcal{G}(\bx,t;\boldsymbol\xi,\tau)$ is given by~\eqref{green_function}.

\section{Estimation of the source term}\label{sec:inverse_prob}

Formula~\eqref{sol_diffusion} gives the solution of problem~\eqref{diff_problem}, where $f$ is a given source function. In our case, this source term $f$ is not known, but it has to mimic the presence of the exchangers. In this way, solution~\eqref{sol_diffusion} also provides an approximation of the solution of problem~\eqref{global_problem} and it can be used in the optimisation of the geothermal field. In Section~\ref{subsec:source_term}, we propose a method to obtain the source term as the solution of an inverse source problem. In Section~\ref{subsec:approx_an_sol}, we illustrate a procedure, based on formula~\eqref{sol_diffusion}, for the explicit computation of the approximated solution of problem~\eqref{global_problem}. In Section~\ref{subsec:source_bc_cfr}, we show a comparison among the exact solution of problem~\eqref{global_problem} and the solution of the corresponding source problem~\eqref{diff_problem} in the case of a single exchanger, in order to supply a numerical evidence of the close equivalence between the two formulations.

\subsection{The computation of the source term}\label{subsec:source_term}

The source term in problem~\eqref{diff_problem} has to yield the thermal effect of the exchangers in the surrounding soil. More specifically, this source term has to produce a temperature distribution on the exchangers position that corresponds to the one on the exchangers wall, and a temperature distribution on the free soil that corresponds to the temperature profile of the undisturbed soil. To this aim, we study two similar inverse source problems for the computation of the complete source. In fact, two types of sources are used: the exchanger source $\ve$, with support on the exchanger $E=\{\bx=(x,y,z)^T \in \mathbb{R}^3: x,y\in[-L_E/2,L_E/2],z\in[0,H]\}$, which has to produce the temperature surplus $T_E-T_S$ on the exchanger; the free soil source $\sigma$, with support on the whole domain $\Omega$, which has to produce the temperature profile of the free soil $T_S$. Moreover, the contribution of each source $\ve,\sigma$ in the complete source $f$ is:
\begin{equation}\label{complete_source_def}
f(\bx,t)=\ve(\bx,t)+\sigma(\bx,t), \quad \bx \in \Omega,t \in (0,\overline t).
\end{equation}
Note that the temperature of the exchanger depends on the surrounding soil temperature, and vice-versa, since there is a mutual influence that must be taken into account. We consider the following problems with source terms $\ve$ and $\sigma$:
\begin{equation}\label{sourceE_rod}
\begin{cases}\displaystyle
\frac{\pt e}{\pt t} (\bx,t) - \alpha \Delta e(\bx,t) = \ve(\bx,t), & \bx \in \Omega, t \in (0,\overline t), \\
e(\bx,0) = g_e(\bx), & \bx \in \Omega, \\
e(\bx_0,t) = 0, & \bx_0=(x,y,0)^T, t \in (0,\overline t), \\
e(\bx_H,t) = 0, & \bx_H=(x,y,H)^T, t \in (0,\overline t),
\end{cases}
\end{equation}
\begin{equation}\label{sourceS_rod}
\begin{cases}\displaystyle
\frac{\pt u}{\pt t} (\bx,t) - \alpha \Delta u(\bx,t) = \sigma(\bx,t), & \bx \in \Omega, t \in (0,\overline t), \\
u(\bx,0) = g_s(\bx), & \bx \in \Omega, \\
u(\bx_0,t) = T_0, & \bx_0=(x,y,0)^T, t \in (0,\overline t), \\
u(\bx_H,t) = T_H, & \bx_H=(x,y,H)^T, t \in (0,\overline t),
\end{cases}
\end{equation}
where the functions in the initial conditions
\begin{equation}\label{initial_cond_source}
g_e(\bx) = 0, \qquad g_s(\bx)=\left(T_H-T_{0}\right) \frac{z}{H}+T_{0},
\end{equation}
are specifically chosen to provide a simplification in the analytical solution of the problems. From standard arguments on the heat equation, we have that the initial condition does not influence the solution for large values of $t>0$. In particular, for our problems a characteristic value of $\overline t$ is $100$ days (cold/warm season length) and the contribution of the initial term is negligible after $5$ days. Thus, we can suppose that the solutions of problems~\eqref{sourceE_rod},\eqref{sourceS_rod} are independent from the initial condition. The source terms $\ve(\bx,t)$, $\sigma(\bx,t), \, \bx \in \Omega, t \in (0,\overline t)$ have to produce the following result
\begin{alignat}{2}
e(\bx,t) &= T_E(z,t)-T_S(z), \quad && \bx \in E, \label{usol_sourceE} \\
u(\bx,t) &= T_S(z), \quad && \bx \in \Omega, \label{usol_sourceS}
\end{alignat}
for problems~\eqref{sourceE_rod} and~\eqref{sourceS_rod}, respectively. The choice of the solutions $e,u$ in Equations~\eqref{usol_sourceE},\eqref{usol_sourceS}, respectively, reveals that the soil source $\sigma$ acts in any point of the domain $\Omega$, exchangers included, while the exchanger source $\ve$ acts only in the points corresponding to the device position. We note that this choice for the source terms considers an exchanger placed at the domain center with respect to the $x$-$y$ plane. Being the profiles $T_E,T_S$ independent of $x,y$, we assume that the source terms $\ve,\sigma$ have the same symmetry, thus $\ve$ is valid for every exchanger in the field regardless of its position on the $x$-$y$ plane. More precisely, the effect of an exchanger at $(\overline x,\overline y)$ can be produced by $\ve(x-\overline x, y-\overline y,z,t)$; similar arguments hold for the source $\sigma$. This fundamental property is a simple consequence of the convolution form of $\mathcal{G}$ in~\eqref{green_function}, and allows us to avoid the computation of the sources at each point of the domain $\Omega$, which would make the implementation of the proposed method computationally unaffordable.

We propose the following procedure to compute the solutions $\ve,\sigma$ of equations~\eqref{usol_sourceE},\eqref{usol_sourceS}, respectively. From formula~\eqref{sol_diffusion} we have that the solution of problem~\eqref{sourceE_rod} at $x=0$, $y=0$ (position of the exchanger) is
\begin{equation}\label{sourceE_eq}
\begin{split}
e(0,0,z,t) = & \sum_{r=1}^{\infty} \sin \left(\frac{r\pi z}{H} \right) \int_0^{t} e^{-\frac{r^2\pi^2 \alpha (t-\tau)}{H^2}} \left(\int_{-L_E/2}^{L_E/2} \frac{e^{-\frac{\xi^2}{4\alpha(t-\tau)}}}{\sqrt{2\alpha\pi H (t-\tau)}} d\xi \right.\\
& \left. \int_{-L_E/2}^{L_E/2} \frac{e^{-\frac{\eta^2}{4\alpha(t-\tau)}}}{\sqrt{2\alpha\pi H (t-\tau)}} d\eta \int_0^H \sin \left(\frac{r\pi \zeta}{H} \right) \ve(\zeta,\tau) d\zeta \right) d\tau,
\end{split}
\end{equation}
where the space integrals in the $x$ and $y$ variables have been restricted to the horizontal cross section of exchanger $E$, assuming that $E$ is the support of the source term $\ve$. In a similar way the solution $u$ of problem~\eqref{sourceS_rod} can be translated by function $g_s$ and represented by
\begin{equation}\label{sourceS_eq}
\begin{split}
s(0,0,z,t) = & \ u(0,0,z,t)-g_s(z) \\
= & \sum_{r=1}^{\infty} \sin \left(\frac{r\pi z}{H} \right) \int_0^{t} e^{-\frac{r^2\pi^2 \alpha (t-\tau)}{H^2}} \left(\int_{-A}^{A} \frac{e^{-\frac{\xi^2}{4\alpha(t-\tau)}}}{\sqrt{2\alpha\pi H (t-\tau)}} d\xi \right.\\
& \left. \int_{-B}^{B} \frac{e^{-\frac{\eta^2}{4\alpha(t-\tau)}}}{\sqrt{2\alpha\pi H (t-\tau)}} d\eta \int_0^H \sin \left(\frac{r\pi \zeta}{H} \right) \sigma(\zeta,\tau) d\zeta \right) d\tau,
\end{split}
\end{equation}
where we assumed that the support of the source term $\sigma$ is $\hat D=(-A,A)\times(-B,B)\times(0,H)$, that is the finite approximation of $\Omega$. Formulas~\eqref{sourceE_eq} and~\eqref{sourceS_eq} define explicit relations between the source terms $\ve,\sigma$ and the solutions $e,s$ of problems~\eqref{sourceE_rod},\eqref{sourceS_rod}, respectively; on the contrary, they are integral equations for $\ve$ and $\sigma$ when the functions $e$ and $s$ are known. We propose a simple procedure to obtain the solution of such integral equations that is based on the Fourier series expansion. In particular, on the basis of the boundary conditions, we consider the Fourier sine series expansion of $e$ in [0,H],
\begin{equation}\label{expansionE}
e(0,0,z,t)=\sum_{r=1}^{\infty} e_r(t) \sin \left( \frac{r\pi z}{H} \right),
\end{equation}
whose coefficients $e_r, r=1,2,\dots$, are given by
\begin{equation}\label{coeffE}
e_r(t) = \frac{2}{H} \int_0^H e(0,0,z,t) \sin \left( \frac{r\pi z}{H} \right)dz.
\end{equation}
From equation~\eqref{sourceE_eq} and using the Fourier expansion~\eqref{expansionE}, we obtain
\begin{equation}\label{integ_eqE}
\begin{split}
e_r(t) = & \int_0^{t} e^{-\frac{r^2\pi^2 \alpha (t-\tau)}{H^2}} \ve_r(\tau) \left(\int_{-L_E/2}^{L_E/2} \frac{e^{-\frac{\xi^2}{4\alpha(t-\tau)}}}{\sqrt{2\alpha\pi H (t-\tau)}} d\xi \right. \\
& \left. \int_{-L_E/2}^{L_E/2} \frac{e^{-\frac{\eta^2}{4\alpha(t-\tau)}}}{\sqrt{2\alpha\pi H (t-\tau)}} d\eta \right) d\tau, \quad r=1,2,\dots,
\end{split}
\end{equation}
where
\begin{equation}\label{fourier_exch_notdiscr}
\varepsilon_{r} (\tau) = \int_0^H \ve(\zeta,\tau) \sin \left(\frac{r\pi \zeta}{H} \right) d\zeta.
\end{equation}
Volterra integral equations~\eqref{integ_eqE} are considered for $t\in[0,\overline t]$ and are solved numerically by a simple quadrature method. We divide the time interval $[0,\overline{t}]$ into $N_t$ subintervals of length $\Delta t$. So, the time variable $t$ in discretised form becomes $t_n=n\Delta t$, $n=1,\dots,N_t$, and the discretisation nodes of the time variable $\tau$ are $\tau_p=(p-0.5)\Delta t$, $p=1,\dots,N_t$. From formula~\eqref{integ_eqE} by the midpoint quadrature rule, we obtain the following linear system for the source term coefficients:
\begin{equation}\label{sysE}
\Delta t \sum_{p=1}^{n} \ce_{n,p}^r \ve_{r}(\tau_p) = e_r(t_n), \qquad n=1,\dots,N_t, r=1,\dots,R,
\end{equation}
where $R$ is the truncation index in the Fourier series, $\ve_r(\tau_p)$ are the unknowns, $e_r(t_n)$ are the known terms given by formula~\eqref{coeffE}, and the matrix $\ce^r$ contains the approximation of the exponential time integral and the exact solutions of the space integrals with respect to $x$ and $y$, i.e.,
\begin{equation*}
\ce_{n,p}^r = \frac{2}{H} e^{-\frac{r^2\pi^2 \alpha (n-p+0.5) \Delta t}{H^2}} \erf^2\left(\frac{L_E}{4\sqrt{\alpha (n-p+0.5) \Delta t}}\right),
\end{equation*}
where $\erf$ is the Gauss error function.

An analogous procedure can be applied to $s$, thus integral equation~\eqref{sourceS_eq} is discretised by the following linear system:
\begin{equation}\label{sysS}
\Delta t \sum_{p=1}^{n} \cs_{n,p}^r \sigma_{r}(\tau_p) = s_r(t_n), \qquad n=1,\dots,N_t, r=1,\dots,R, 
\end{equation}
where
\begin{equation}\label{fourier_soil}
\sigma_{r}(\tau) = \int_0^H \sigma(\zeta,\tau) \sin \left(\frac{r\pi \zeta}{H} \right) d\zeta,
\end{equation}
and 
\begin{equation*}
\cs_{n,p}^r = \frac{2}{H} e^{-\frac{r^2\pi^2 \alpha (n-p+0.5) \Delta t}{H^2}} \erf\left(\frac{A}{2\sqrt{\alpha (n-p+0.5) \Delta t}}\right) \erf\left(\frac{B}{2\sqrt{\alpha (n-p+0.5) \Delta t}}\right).
\end{equation*}
Systems~\eqref{sysE},\eqref{sysS} are solved by means of $LU$ factorisation. We note that an important property of the discretisation schemes~\eqref{sysE},\eqref{sysS} is that they are decoupled with respect to the variable $r$, so the solution for each $r=1,\dots,R$ can be computed independently from the other ones.

\subsection{The approximated solution for the BHE field}\label{subsec:approx_an_sol}

Once the source term has been estimated, formula~\eqref{sol_diffusion} can be used to compute the solution of problem~\eqref{diff_problem}. We choose the initial temperature distribution $g(\bx), \, \bx \in \hat D,$ equal to the function $w(z)$ defined in~\eqref{translation}; as already mentioned, this choice does not influence the temperature $u$ for large values of $t$. For later convenience, we want to emphasise the dependence of $u$ on the exchangers position, so we denote with
$\bp=\left(p_1,p_2,\dots,p_{2N_E}\right )\in \mathbb{R}^{2N_E},$
the vector of the $x,y$ coordinates of the exchangers centre. Thus, the solution of~\eqref{diff_problem} becomes
\begin{equation}\label{sol_diffusion_onD}
\begin{split}
u(\bx,t;\bp) = & \int_0^{t} \int_{\hat D} \mathcal{G}(\bx,t;\boldsymbol\xi,\tau)f(\boldsymbol\xi,\tau)d\boldsymbol\xi d\tau + \frac{z}{H}(T_H-T_0) +T_0 \\
= & \int_0^{t} \left( \sum_{l=1}^{N_E} \int_{E_l} \mathcal{G}(\bx,t;\boldsymbol\xi,\tau)\ve(\boldsymbol\xi,\tau)d\boldsymbol\xi + \int_{\hat D} \mathcal{G}(\bx,t;\boldsymbol\xi,\tau)\sigma(\boldsymbol\xi,\tau)d\boldsymbol\xi \right) d\tau + \frac{z}{H}(T_H-T_0) +T_0 \\
= & \sum_{r} \sin\left(\frac{r \pi z}{H}\right) \int_{0}^{t} e^{-\frac{r^2 \pi^2 \alpha (t-\tau)}{H^2}} \Bigg( \sum_{l=1}^{N_E} \varphi(x-p_{2l-1},t-\tau) \varphi(y-p_{2l},t-\tau) \ve_r(\tau) \Bigg.\\
& \Bigg.+ \frac{2}{H}\erf\left(\frac{A}{2\sqrt{\alpha (t-\tau)}}\right)\erf\left(\frac{B}{2\sqrt{\alpha(t-\tau)}}\right) \sigma_r(\tau) \Bigg) d\tau+ \frac{z}{H}(T_H-T_0) +T_0,
\end{split}
\end{equation}
where $\bx=(x,y,z)^T \in \hat{D}$, $E_l,N_E,L_E$ have been defined in Section~\ref{sec:generic_form}, $\ve_r,\sigma_r$ have been defined in~\eqref{fourier_exch_notdiscr},\eqref{fourier_soil}, and
\begin{equation}\label{fun_erf}
\varphi(q,s) = \frac{1}{\sqrt{2H}}\left[\erf\left(\frac{1}{2\sqrt{\alpha s}}\left(\frac{L_E}{2}+q \right)\right) + \erf\left(\frac{1}{2\sqrt{\alpha s}}\left(\frac{L_E}{2}-q \right)\right)\right].
\end{equation}
We note that the source term $f(\boldsymbol\xi,\tau)$ in~\eqref{sol_diffusion_onD} actually does not depend on $\xi,\eta$ but only on $\zeta,\tau$ due to the simplifying assumption made in Section~\ref{subsec:source_term} about the independence of the source terms from the position of the exchangers. Moreover, the source has been split into $\ve$ (i.e., its Fourier coefficients $\ve_r$) that accounts for the exchanger action and $\sigma$ (i.e., its Fourier coefficients $\sigma_r$) that accounts for the soil action, according to Section~\ref{subsec:source_term}. Consequently, we split the spatial integral into two parts: one over the exchangers support and the other over the soil support that is the whole $\hat D$. {Another important consideration on the computation of $u$ is that the source terms $\ve,\sigma$ neglect the interference among devices, which is instead considered in $u$ by means of the superposition of functions $\varphi$.} Formula~\eqref{sol_diffusion_onD} is approximated by a simple numerical scheme on $\hat D\times[0,\overline t]$. We describe this approximation for the evaluation of $u$ on a uniform grid; a similar approach can be used to evaluate $u$ on different points. The time grids for $t$ and $\tau$ are those defined in Section~\ref{subsec:source_term}. The variable $\bx$ is discretised as in Section~\ref{sec:numeric_he}, i.e., by the grid points $(x_i,y_j,z_k)^T=(-A+i h_x,-B+j h_y, k h_z)^T$, $(i,j,k)\in I \cup B_S \cup B_E$. Applying the midpoint rule to the time integral, the final form of the discretised solution is
\begin{equation}\label{sol_an_discr}
\begin{split}
u_{i,j,k}^n = & \sum_{r=1}^R \sin\left(\frac{r \pi z_k}{H}\right) \Delta t \sum_{p=1}^n \mathscr{T}_{n-p,r} \phantom{\sum_{jj=1}^{N_E}} \\
& \left( \sum_{l=1}^{N_E} \varphi(x_i-p_{2l-1},t_{n-p}) \varphi(y_j-p_{2l},t_{n-p}) \ve_r(\tau_p) + {\mathscr{S}}_{n-p} \sigma_r(	\tau_p)\right)\\
& +\frac{z_k}{H}(T_H-T_0)+T_0,
\end{split}
\end{equation}
for all $i=0,\dots,N_x, \, j=0,\dots,N_y, \, k=0,\dots,N_z, \, n=1,\dots,N_t$, where $t_{n-p} = (n-p+0.5) \Delta t$, and the matrixes $\mathscr{T},\mathscr{S}$ are defined as follows,
\begin{align*} 
& \displaystyle \mathscr{T}_{n-p,r} = e^{-\frac{r^2\pi^2 \alpha t_{n-p} }{H^2}}, \\
& \displaystyle \mathscr{S}_{n-p} = \frac{2}{H}\erf\left(\frac{A}{2\sqrt{\alpha t_{n-p}}}\right)\erf\left(\frac{B}{2\sqrt{\alpha t_{n-p}}}\right).
\end{align*}

\subsection{Equivalence between heat conduction problems}\label{subsec:source_bc_cfr}

The physical phenomena described by problem~\eqref{global_problem} and problem~\eqref{diff_problem} are quite different. In fact, in problem~\eqref{global_problem} exchangers are outside the domain and their effect is considered by means of proper boundary conditions at the interface between the soil and the exchangers; whereas in problem~\eqref{diff_problem} exchangers are part of the domain and their effect is reproduced as a heat source. Nevertheless, these two heat conduction problems can be equivalent; here, the equivalence of the two problems is based on the soil temperature generated by their solution. Now, we compare these two formulations in order to give a numerical evidence of their equivalence. For the sake of simplicity, this comparison deals with a single exchanger placed in a sufficiently large soil portion. This simple study is able to show that the effect of the exchanger on the soil temperature, which should be formulated as problem~\eqref{global_problem}, is similar to the one of a proper heat source. In other words, the agreement of the results given by the two formulations allows the use of problem~\eqref{diff_problem} to describe the heat transfer into a geothermal field, which has the advantage to make much more easier the solution computation especially with several exchangers.

In the case of a single exchanger, problem~\eqref{global_problem} can be stated in cylindrical coordinates, that is
\begin{equation}\label{cyl_diff}
\begin{cases}\displaystyle
\frac{1}{\alpha} \frac{\pt w}{\pt t} (r,z,t) - \frac{\pt^2 w}{\pt r^2} (r,z,t) - \frac{1}{r} \frac{\pt w}{\pt r} (r,z,t) - \frac{\pt^2 w}{\pt z^2} (r,z,t) = 0, & \begin{split} &r>a, 0<z<H, \\ & t \in (0,\overline t),\end{split} \\
w(r,z,0) = T_S(z), & r>a, 0<z<H, \\
w(a,z,t) = T_E(z), & 0<z<H, t \in (0,\overline t), \\
\lim_{r \rightarrow +\infty} w(r,z,t)=T_S(z), & 0<z<H, t \in (0,\overline t), \\
w(r,0,t) = T_S(0), & r>a, t \in (0,\overline t), \\
w(r,H,t) = T_S(H), & r>a, t \in (0,\overline t),
\end{cases}
\end{equation}
where the exchanger has circular cross section with radius $a=L_E/2$. We solve problem~\eqref{cyl_diff} as done in Section~\ref{sec:analytic_he}, applying the translation $T_S(z)$ and the Green's function approach. The Green's function in the spatial domain outside a cylindrical hole is obtained by a multiplication theorem for cylindrical Green's functions similar to Theorem~\ref{separable_domain}. In fact, if the problem has azimuthal symmetry, the multidimensional Green's function can be obtained by multiplying one-dimensional Green's functions~\cite{ozisik,heatGreenfun}, that is
\begin{equation*}
\begin{split}
\mathcal{G^C} (r,z,t;r',\zeta,\tau) = & \frac{1}{\pi a^2 H} \int_{0}^{\infty} \beta e^{-\frac{\beta^2 \alpha (t-\tau)}{a^2}} \phi(\beta,r) \frac{\phi(\beta,r')}{J_0^2\left(\beta\right)+Y_0^2\left(\beta\right)} \, d\beta \\
& \sum_{l=1}^\infty e^{-\frac{l^2\pi^2 \alpha (t-\tau)}{H^2}} \sin\left(\frac{l \pi z}{H}\right) \sin\left(\frac{l \pi \zeta}{H}\right),
\end{split}
\end{equation*}
where $\phi(\beta,r)=J_0\left({\beta r}/{a}\right)Y_0\left(\beta\right)-Y_0\left({\beta r}/{a}\right)J_0\left(\beta\right)$. Thus it can be shown that the solution of problem~\eqref{cyl_diff} is
\begin{equation}\label{sol_1e_cyl}
\begin{split}
w(r,z,t) = &-\frac{2\alpha}{\pi a^2} \sum_{l=1}^\infty \sin\left(\frac{l \pi z}{H}\right) e_l \int_0^t e^{-\frac{l^2 \pi^2 \alpha (t-\tau)}{H^2}} \\
& \left(\int_{0}^{\infty} \beta e^{-\frac{\beta^2 \alpha (t-\tau)}{a^2}} \frac{\phi(\beta,r)}{J_0^2\left(\beta\right)+Y_0^2\left(\beta\right)} \, d\beta \right) \, d\tau + T_S(z),
\end{split}
\end{equation}
where $e_l$, $l=1,2,\dots,$ are the Fourier coefficients of the function $T_E-T_S$. See~\cite{ozisik} for the general solution formula.

The alternative formulation is given by problem~\eqref{diff_problem} where the source term accounts for only one exchanger placed at the domain center. Thus, from formula~\eqref{sol_diffusion_onD}, the solution $v$ of this problem is
\begin{equation}\label{sol_1e_src}
\begin{split}
v(\bx,t)= & \sum_{r} \sin\left(\frac{r \pi z}{H}\right) \int_{0}^{t} e^{-\frac{r^2 \pi^2 \alpha (t-\tau)}{H^2}} \Bigg( \varphi(x-p_{1},t-\tau) \varphi(y-p_{2},t-\tau) \ve_r(\tau) \Bigg.\\
& \Bigg.+ \frac{2}{H}\erf\left(\frac{A}{2\sqrt{\alpha (t-\tau)}}\right)\erf\left(\frac{B}{2\sqrt{\alpha(t-\tau)}}\right) \sigma_r(\tau) \Bigg) d\tau+ \frac{z}{H}(T_H-T_0) +T_0,
\end{split}
\end{equation}
where $\bx\in\hat D$, $\bp=(p_1,p_2)=(0,0)$, and $\ve_r,\sigma_r,\varphi$ have been defined in~\eqref{fourier_exch_notdiscr},\eqref{fourier_soil},\eqref{fun_erf}, respectively.

Both solutions $w$ and $v$ are approximated on $\hat D\times[0,\overline t]$ by a numerical scheme similar to the one described at the end of Section~\ref{subsec:approx_an_sol}. In addition, in formula~\eqref{sol_1e_cyl}, the domain of the integral over $\beta$ is truncated at $\beta=5a/\sqrt{\alpha \Delta t}$ and the resulting integral is approximated by the midpoint rule. We denote with $w_{i,j,k}^n$, $v_{i,j,k}^n$ the discretised form of $w$, $v$, respectively, on the uniform grid used in Section~\ref{subsec:approx_an_sol}. These results are reported in Section~\ref{sec:results}.

\section{Optimisation of the exchangers position}\label{sec:optimization}

We propose a method to compute the optimal placement of the exchangers in a geothermal field. In principle, this problem should maximise the energy exchanged with the soil but the corresponding objective function is too much complex for defining a practical design tool. However, a strictly related condition can be obtained by requiring a minimum deviation of the soil temperature with respect to the undisturbed temperature profile. The resulting optimisation process applied to the geothermal system is then made possible by using the relations discussed in Sections~\ref{sec:analytic_he},\ref{sec:inverse_prob}.

The objective function of the proposed optimisation problem is:
\begin{equation}\label{objfun}
F(\bp;\mathcal{X},t) = \sum_{k=1}^{P} \left( T_S(z_k)-u(\bx_k,t;\bp) \right)^2,
\end{equation}
where $\mathcal{X}=\{\bx_k=(x_k,y_k,z_k) \in D, k=1,\ldots,P\}$ is the set of the evaluation points, $P$ is the total number of evaluation points in $D$, $T_S$ is again the undisturbed soil temperature, $u(\bx_k,t;\bp)$ is the soil temperature at point $\bx_k$ and time $t$, defined in~\eqref{sol_diffusion_onD}. We note that having required that the evaluation points lie in $D$ means that $(x_k,y_k)\neq(p_i,p_{i+1})$, for all $k=1,\ldots,P$, $i=1,3,\ldots,2N_E-1$, since we are interested in the evaluation of the temperature outside the exchangers. Moreover, the time interval where we search for the optimal arrangement can be in principle different from the time interval $[0,\overline t]$ where we need to analyse the temperature of the geothermal field. In particular, this optimisation process is considered for a time $t=\tilde t$ with $\tilde t \in (0,\overline t)$. The optimisation problem consists in the minimisation of the objective function~\eqref{objfun} constrained to $\mathcal{C}$, i.e.,
\begin{equation}\label{min_problem}
\min_{\bp \in \mathcal{C}} F(\bp;\mathcal{X},\tilde t),
\end{equation}
where the constraints are given by
\begin{equation*}
\mathcal{C} = \left\{ \bp = \left(p_1,p_2,\ldots,p_{2N_E}\right): \ (p_{i},p_{i+1}) \in [-A,A] \times [-B,B], i=1,3,\dots,2N_E-1 \right\}.
\end{equation*}
The solution of problem~\eqref{min_problem} is computed by using the steepest descent method, that is
\begin{equation}\label{grad_descent}
\bp^{n+1}=\bp^n-\beta \, \nabla F(\bp^n;\mathcal{X},\tilde t),
\end{equation}
where $\nabla$ is the gradient operator with respect to variables $\bp$, $n \in \bbn$ is the algorithm step, with $\bp^0$ the initial position of the exchangers, $\beta$ is the parameter corresponding to the descent step length computed by a simple line search procedure. An active set strategy~\cite{practicalOptimization} is used to deal with the box constraints. We note that this is the simplest method among the class of gradient approaches, but it allows us to observe another important property of function $F$: its gradient can be calculated analytically avoiding the finite difference approximation; in fact, for $i=1,3,\ldots,2N_E-1$, we have:
\begin{equation}\label{F_deriv}
\begin{split}
\frac{\pt F}{\pt p_i}(\bp;\mathcal{X},\tilde t) & = -2 \sum_{k=1}^{P} \left( T_S(z_k)-u(\bx_k,\tilde t;\bp) \right) \frac{\pt u}{\pt p_i}(\bx_k,\tilde t;\bp) \\
& = -2 \sum_{k=1}^{P} \left( T_S(z_k)-u(\bx_k,\tilde t;\bp) \right) \left( \sum_{k'} \sin\left(\frac{k' \pi z_k}{H}\right) \int_{0}^{\tilde t} e^{-\frac{k'^2 \pi^2 \alpha (\tilde t-\tau)}{H^2}} \right. \\
& \left. \phantom{\int_{0}^{\tilde t}} \frac{\pt}{\pt p_i} \varphi(x_k-p_i,\tilde t-\tau) \varphi(y_k-p_{i+1},\tilde t-\tau) \ve_{k'}(\tau) d\tau \right),
\end{split}
\end{equation}
where $\varphi$ has been defined in~\eqref{fun_erf}, while $\ve_{k'}$ has been defined in~\eqref{fourier_exch_notdiscr}, and
\begin{equation}\label{erf_deriv}
\frac{\pt}{\pt p_i} \varphi(x_k-p_i,\tilde t-\tau) = \frac{1}{\sqrt{2\pi \alpha H (\tilde t-\tau)}} \left[e^{-\frac{\left(L_E/2+(x_k-p_i)\right)^2}{4\alpha(\tilde t-\tau)}} - e^{-\frac{\left(L_E/2-(x_k-p_i)\right)^2}{4\alpha(\tilde t-\tau)}} \right].
\end{equation}
The derivative with respect to $p_{i+1}$ can be computed by using a similar formula. The time integral in formula~\eqref{F_deriv} is approximated by the midpoint rule. In addition to the analytical computation of the derivative, our method is also based on the choice of the evaluation points in set $\mathcal{X}$ that must be well separated from the position of the exchangers. A simple strategy for this choice is described in Section~\ref{subsec:results}. Finally, the algorithm~\eqref{grad_descent} stops when one of the following criteria is fulfilled: 
\begin{equation}\label{stop_crit}
|\bp^{n+1}-\bp^n|_\infty\leq tol_1 \quad \te{ or } \quad \log_{10}(F(\bp^0;\mathcal{X},t))-\log_{10}(F(\bp^n;\mathcal{X},t))\geq tol_2 \quad \te{ or } \quad n\geq \overline n,
\end{equation}
where $tol_1,tol_2 \in \mathbb{R} \te{ and } \overline n \in \bbn$ are prescribed tolerances monitoring the distance between two successive solutions, the improvement in the objective function and the maximum number of allowed steps, respectively.

\section{Results}\label{sec:results}

We present some results obtained by numerical experiments to test the methods proposed in the previous sections, both for the analytical solution of the heat diffusion problem in a geothermal field and for the solution of the corresponding optimisation problem. The general context taken into account is a geothermal field operating in winter mode for about six months. As the geothermal plant is started, a conductive heat flux occurs from the warmer soil to the colder exchangers thus the soil temperature profile is modified by this heat loss, in turn the exchanger profile is modified by this heat gain. With the following simulations, we evaluate how the soil temperature field is affected by the presence of the exchangers. In Section~\ref{subsec:config}, the physical and geometrical characterisation of the geothermal field is given. In Section~\ref{subsec:settings}, the simulation details are described. Finally, in Section~\ref{subsec:results}, we show and discuss results obtained with the proposed methods.

\subsection{The exchanger convective problem and data of the geothermal field}\label{subsec:config}

The temperature profile $T_E$ at the exchanger wall, i.e., $z \in [0,H_E]$, has been obtained by exploiting the theory of completely developed fluid flows inside rectilinear pipes~\cite{bejan}. This hypothesis is reasonably valid for U-shaped exchangers, that constitute a simplified case but also one of the most common; other strategies should be considered for exchangers having different shapes. The convective thermal exchange between the fluid and the surrounding soil undergoes the first principle of Thermodynamics, that is
\begin{equation}\label{heat_dev_flow}
\begin{cases}
\displaystyle\frac{dT}{dz}(z,t)=\frac{k Nu}{b^2\rho c_p U}\left(u(\bx,t;\bp)-T(z,t)\right), & 0<z<H_E, \\
T(\overline z,t)=\overline T(t), & \overline z \in [0,H_E],
\end{cases}
\end{equation}
where $T$ is the mean temperature on the cross section of the pipe; $t$ is a fixed time instant in $[0,\overline t]$; $\bx=(x,y,z)^T$ is any point on the exchanger wall since there is azimuthal symmetry in the soil temperature; $k$ is the thermal conductivity of the fluid; $b$ is the pipe radius; $\rho,c_p$ are the density and the specific heat of the fluid, respectively; $U$ is the mean velocity of the fluid flow and $Nu$ is the Nusselt number. Cauchy problem~\eqref{heat_dev_flow} holds for the fluid in the descending and ascending parts of the pipe; more precisely, we have: $T=T_E^d$, $\overline z=0$, $\overline T(t)=T_{in}$ is the inlet fluid temperature, when considering the descending pipe; $T=T_E^u$, $\overline z=H_E$, $\overline T(t)=T_E^d(H,t)$, when considering the ascending pipe. So, in this evaluation, we neglect the contribution provided by the small curve at the bottom of the exchanger. From standard arguments on ordinary differential equations, we have:
\begin{equation*}
T_E^d(z,t)=e^{-cz} \left(T_{in} + c\int_0^z u(x,y,\zeta,t) e^{c\zeta} \, d\zeta \right),
\end{equation*}
\begin{equation*}
T_E^u(z,t)=e^{-c(z-H_E)} \left(T_E^d(H_E,t) + c\int_H^z u(x,y,\zeta,t) e^{c(\zeta-H_E)} \, d\zeta \right),
\end{equation*}
where $c=k Nu/b^2\rho c_p U$. As a consequence, the temperature profile $T_E$ on the whole interval $[0,H]$ is given by the average between the ascending and descending fluid temperatures in $[0,H_E]$ and joined continuously to the bottom soil temperature in $(H_E,H]$, that is
\begin{equation}\label{defT_E}
T_E(z,t)=
\begin{dcases}
\displaystyle\frac{T_E^d(z,t)+T_E^u(z,t)}{2}, & \te{if } 0 \leq z \leq H_E, \\
\displaystyle(z-H_E)\frac{T_H-T_{fE}}{H-H_E}+T_{fE}, & \te{if } H_E < z \leq H,
\end{dcases}
\end{equation}
where $T_{fE}=T_E^d(H_E,t)$ is the temperature value at the bottom of the exchangers. On the other hand, the temperature profile of the undisturbed soil given by $T_S$ is
\begin{equation}\label{defT_S}
T_S(z)=
\begin{cases}
\displaystyle\left(T_H-T_0\right) \frac{z}{\overline{H}}+T_0, & \text{if } 0 \leq z \leq \overline H, \\
T_H, & \te{if } \overline H < z \leq H,
\end{cases}
\end {equation}
where $\overline{H}$ is the maximum depth at which the soil temperature is expected to be influenced by seasonal variations. The soil temperature is assumed to linearly increase until depth $\overline H$ and then to have a constant value, which corresponds to the mean value in pelitic and pelitic-arenaceous lithotypes in the central Italy~\cite{parametersHE}. Note that the geothermal gradient is neglected since its effect is irrelevant for the depth of the geothermal fields under consideration.

In the numerical experiments, we considered a space domain $\hat D$ having size $2A=70$ m, $2B=70$ m and $H=40$ m. The exchangers have depth $H_E=25$ m, side length $L_E$ of about $0.25$ m, and pipe radius $b=0.016$ m. Assuming the fluid consists in a mixture of water and ethylene glycol~\cite{giacomini2018}, $k=4.1412\cdot 10^{4}$ J/(day m K), $\rho=1.0411\cdot 10^{3}$ kg/m$^3$, $c_p=3.6915\cdot 10^{3}$ J/(kg K). The flow mean velocity has been estimated $U= 3.6288\cdot 10^{4}$ m/day, according to data in~\cite{giacomini2018}. The evaluation of $Nu$ is obtained by $Nu=0.012\left(Re_D^{0.87}-280\right)Pr^{0.4}$, where $Pr$ is the Prandtl number and $Re_D$ is the Reynolds number for pipe flows, see~\cite{1exGlobalOurs,bejan} for a detailed discussion on empirical laws for $Nu$. The soil thermal diffusivity $\alpha$ is $0.05$ m$^2/$day, unless otherwise specified, $T_0=280$ K, $T_H=286$ K, $\overline{H}=15$ m. The temperature in the geothermal field has been simulated for a time interval of $180$ days, that is $\overline t= 180$ days in problem~\eqref{global_problem}, which in principle corresponds to the six-month period when the exchangers have winter operational mode. Finally, in the optimisation problem, $\tdt=120$ days has been used.

\subsection{Settings of the numerical and analytical solutions}\label{subsec:settings}

The numerical solution of problem~\eqref{global_problem} is computed by the finite difference scheme~\eqref{findiff_explicit} and by the analytical solution~\eqref{sol_an_discr}. The finite difference scheme is based on a uniform space grid with $N_x = 140, N_y = 140, N_z = 80, h_x = h_y = h_z = 0.5$ m. The observation time interval $[0,\overline t]$ has been partitioned in $N_t=225$ time steps with $\Delta t=0.8$ days, so the Courant-Friedrichs-Lewy condition for the three-dimensional heat equation is satisfied. The analytical solution~\eqref{sol_an_discr} is evaluated on the same space and time grid used for the finite difference method. However, it needs the computation of the source terms for the exchangers~\eqref{fourier_exch_notdiscr} and the soil~\eqref{fourier_soil} by means of the linear systems~\eqref{sysE},\eqref{sysS}, respectively. The Fourier coefficients of the exchanger $e_r$ in~\eqref{coeffE} are computed using the fast Fourier transform technique, while the Fourier coefficients of the soil $s_r$ are computed by using the definition from the temperature profile $T_S$~\eqref{defT_S}, obtaining the following result
\begin{equation*}
s_r = \frac{2H(T_H-T_0)}{\overline H r^2 \pi^2}\sin\left(\frac{\overline H r \pi} H\right).
\end{equation*}
We note that the computation of the source terms and the analytical solution cannot be done consequently, in fact, problems~\eqref{global_problem} and~\eqref{heat_dev_flow} are coupled since the soil temperature and the exchanger temperature influence each other. Thus, at each time step we implement a fixed-point iteration on $T_E$ in this way: at the time step $n+1$, from~\eqref{defT_E} the solution $(T_E)_k^{n+1}$ requires the knowledge of $T_E^d$ and $T_E^u$ thus of $u_{i,j,k}^{n+1}$; in turn, the solution $u_{i,j,k}^{n+1}$ in~\eqref{sol_an_discr} requires the knowledge of $(T_E)_k^{n+1}$. Such solutions are available at the previous time step $n$, thus we use $u_{i,j,k}^n$ to first calculate the predictors $(T_E)_k^{n+1,0}$ and consequently $u_{i,j,k}^{n+1,0}$, then we compute the correctors $(T_E)_k^{n+1,\nu}$ and $u_{i,j,k}^{n+1,\nu}$ for $\nu=1,2,\dots$, until a stop criterion on $T_E$ in infinite norm is verified.

\subsection{Numerical results}\label{subsec:results}

\begin{figure}[t!]
\centering
\vspace{10pt}
\subfloat[]{\adjincludegraphics[width=.48\textwidth,clip]{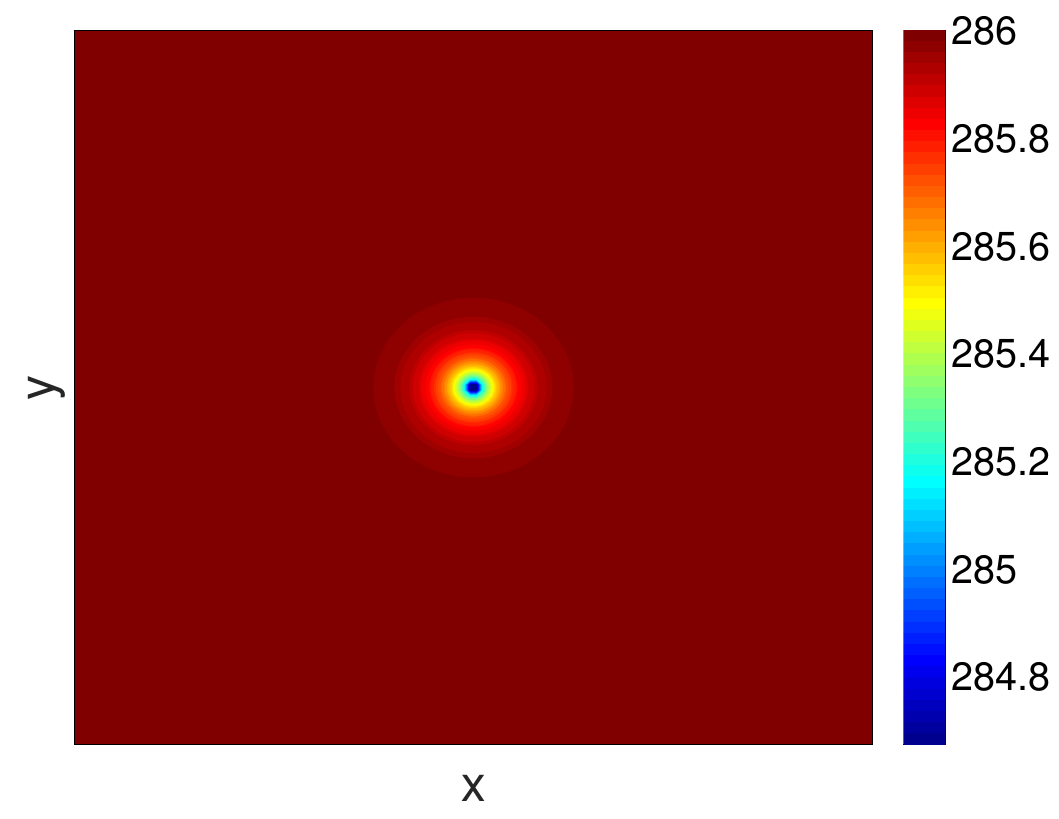}\label{heat_1ex_cfr_bc}}
\subfloat[]{\adjincludegraphics[width=.48\textwidth,clip]{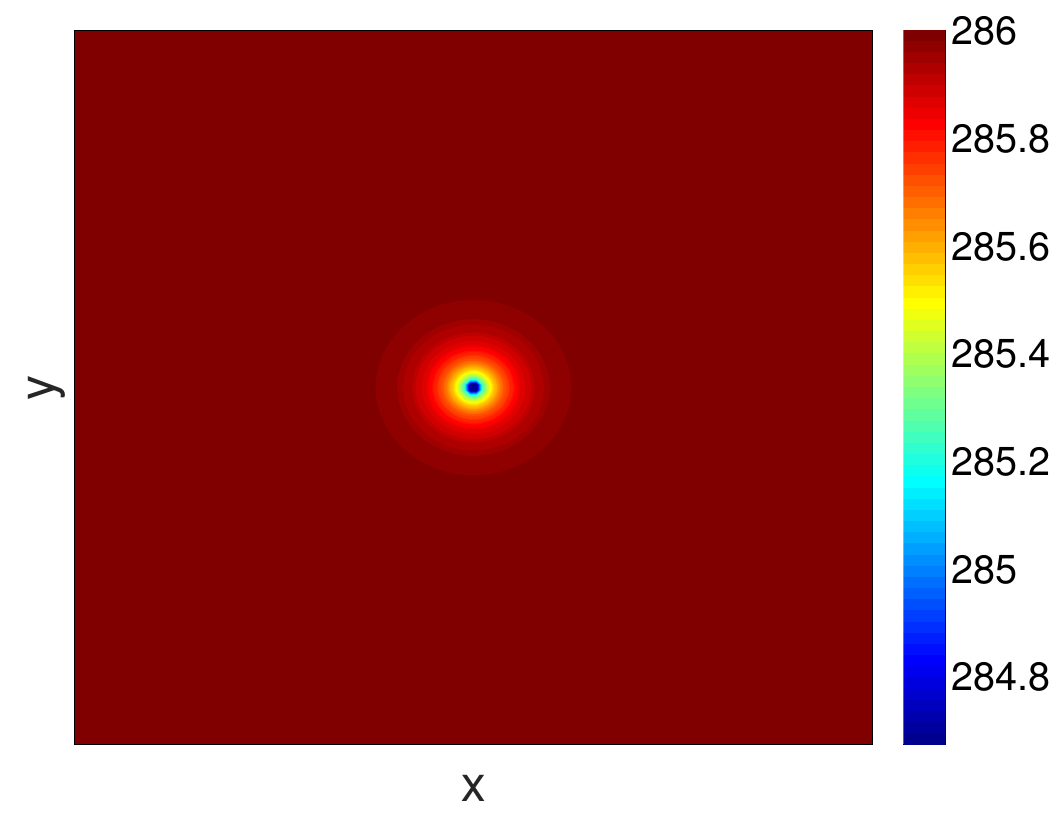}\label{heat_1ex_cfr_src}}
\caption{Temperature distribution after $180$ days around a single exchanger:~\protect\subref{heat_1ex_cfr_bc} solution of the boundary value problem~\eqref{cyl_diff},~\protect\subref{heat_1ex_cfr_src} solution of the forced conduction problem~\eqref{diff_problem}.}\label{heat_1ex_cfr}
\adjincludegraphics[width=.48\textwidth,clip]{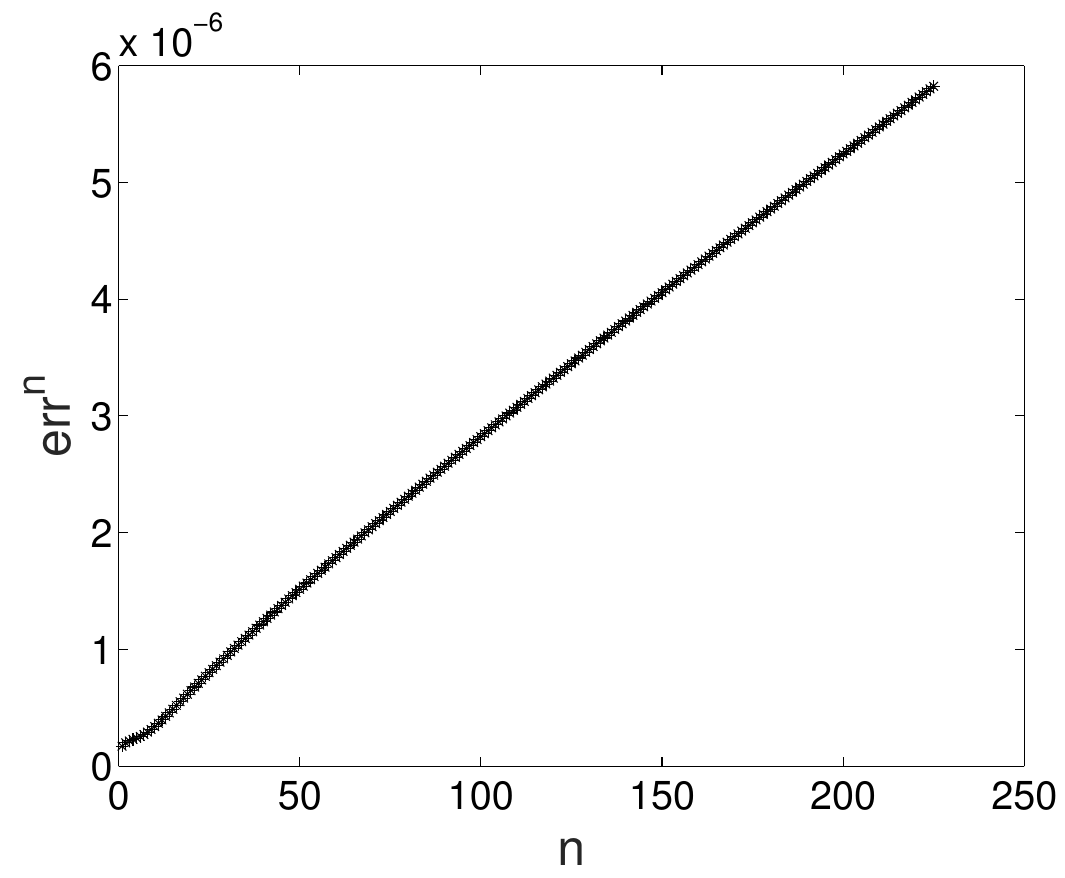}
\caption{Relative error between the solutions of problem~\eqref{diff_problem} and problem~\eqref{cyl_diff} at each time step $n$ of the simulation.}\label{err-heat_1e_cfr}
\end{figure}
We firstly show the results of the comparison between the two formulations of the heat conduction in the geothermal field, see Section~\ref{subsec:source_bc_cfr} for details. Figure~\ref{heat_1ex_cfr} shows the temperature distribution at $z=20$ m in the soil caused by the presence of a single exchanger. In particular, Figure~\ref{heat_1ex_cfr}\subref{heat_1ex_cfr_bc} reports the temperature field~\eqref{sol_1e_cyl} that is solution of problem~\eqref{cyl_diff}, where the effect of the exchanger on the surrounding soil is classically described by a boundary condition on the interface between soil and exchanger. On the other hand, Figure~\ref{heat_1ex_cfr}\subref{heat_1ex_cfr_src} reports the temperature field~\eqref{sol_1e_src} that is solution of problem~\eqref{diff_problem}, where the effect of the exchanger is described by means of a source term in the heat equation. The two temperature fields are almost the same. For a quantitative comparison of these two solutions, we calculate the relative error between the discretised form of $w$ in~\eqref{sol_1e_cyl} and $v$ in~\eqref{sol_1e_src} in norm 1 at each time step, that is
\begin{equation}\label{rel_error}
\displaystyle err^n_k=\frac{\sum_{i=1}^{N_x} \sum_{j=1}^{N_y} |w_{i,j,k}^n-v_{i,j,k}^n|}{\sum_{i=1}^{N_x} \sum_{j=1}^{N_y} |w_{i,j,k}^n|}, \qquad n=1,\dots,N_t,
\end{equation}
where $k=40$ is the index of the half-depth plane (that is at $z=20$ m in these simulations) where the results have been shown by figures. Similar results are obtained for different choices of $z$ due to the symmetry of the problem therefore have not been reported here. Figure~\ref{err-heat_1e_cfr} shows that the relative error is of the order of $10^{-6}$, so the two solutions are in good agreement. The results reported in Figures~\ref{heat_1ex_cfr} and~\ref{err-heat_1e_cfr} gives a numerical evidence of the reliability of the proposed approach to solve analytically problem~\eqref{global_problem} by means of the reformulation~\eqref{diff_problem}. We expect that the two solutions are in good agreement also in the case of multiple exchangers. However, in this case, we cannot solve the heat conduction problem~\eqref{global_problem} by a simple formulation as the one given by problem~\eqref{cyl_diff}; so, a numerical approach must be used to solve this problem, as the one adopted in~\eqref{findiff_explicit}.

\begin{figure}[t!]
\centering
\adjincludegraphics[width=.48\textwidth,clip]{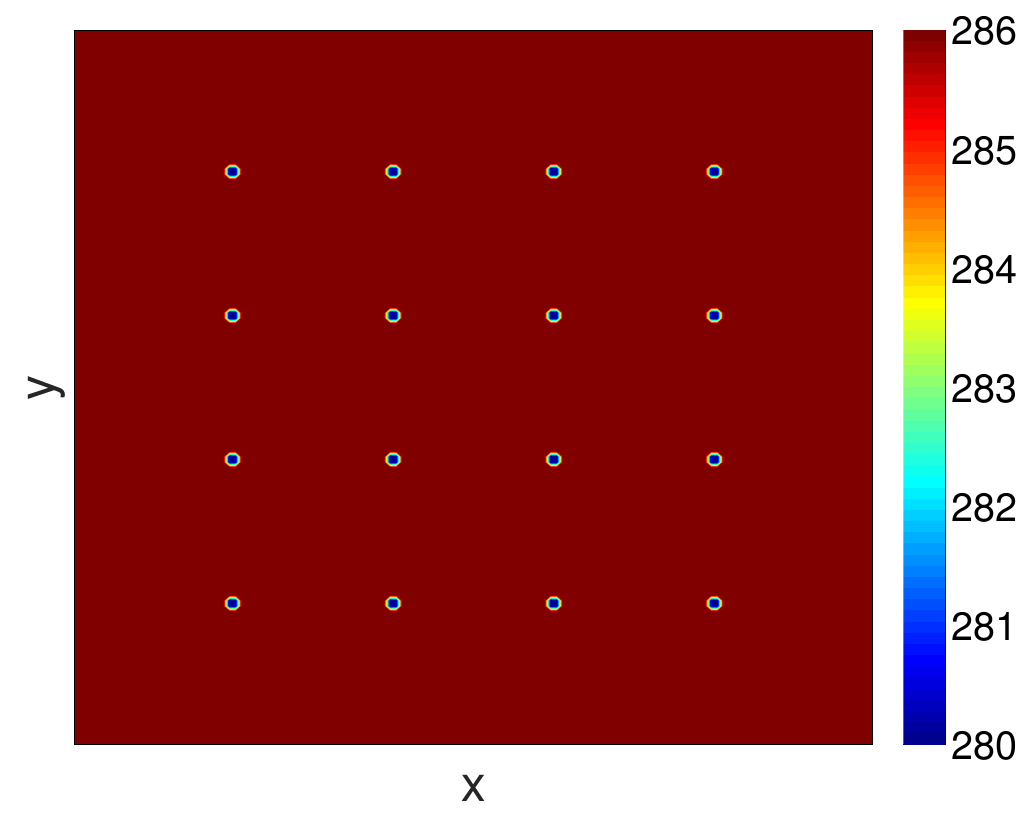}
\caption{Initial temperature distribution for the lattice arrangement with $16$ exchangers.}\label{heat_4x4init_om}
\vspace{10pt}
\subfloat[]{\adjincludegraphics[width=.48\textwidth,clip]{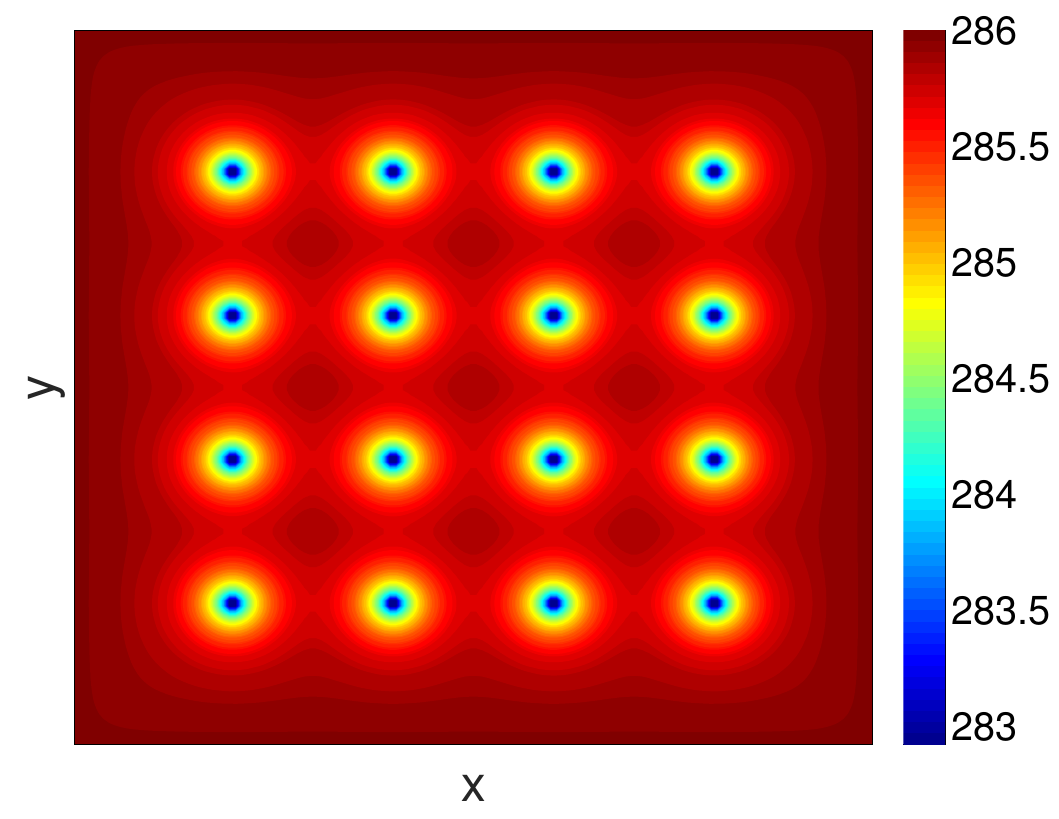}\label{heat_4x4final_om_NUM}}
\subfloat[]{\adjincludegraphics[width=.48\textwidth,clip]{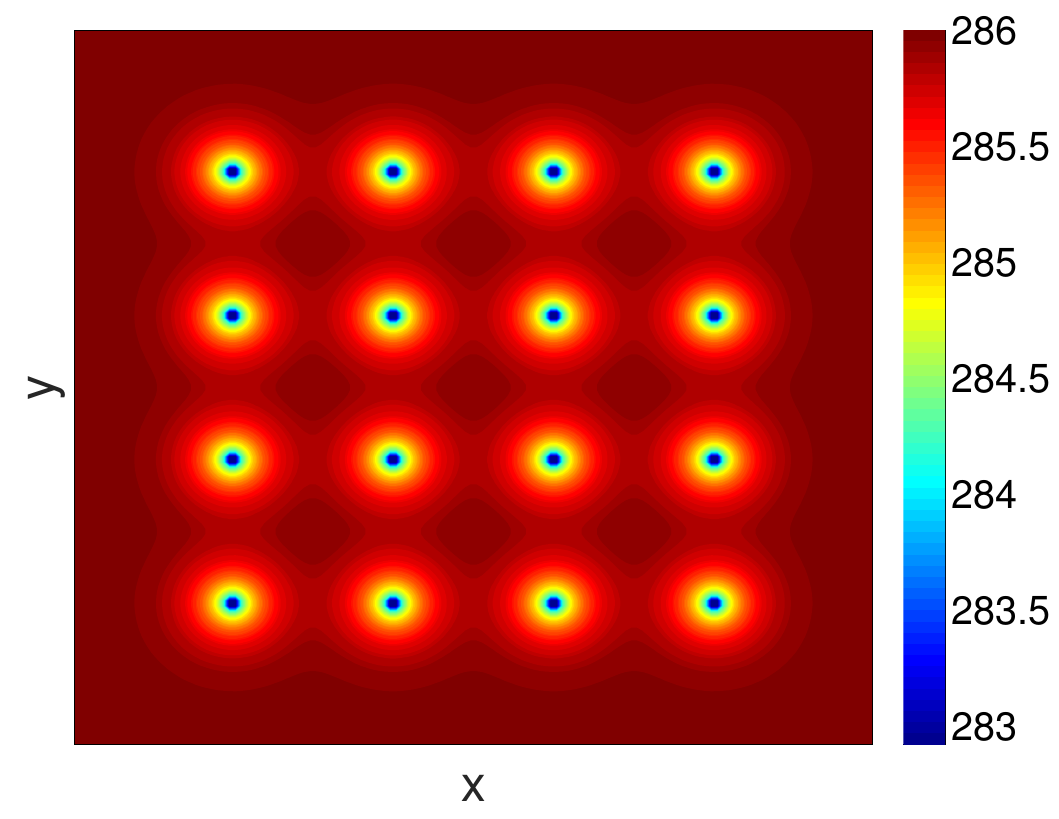}\label{heat_4x4final_om_AN}}
\caption{Temperature distribution after $180$ days for the lattice arrangement with $16$ exchangers:~\protect\subref{heat_4x4final_om_NUM} numerical results,~\protect\subref{heat_4x4final_om_AN} analytical results.}\label{heat_4x4final_om}
\end{figure}
We show the results of the numerical experiment for the comparison between the finite difference solution of the boundary value problem~\eqref{global_problem} and the forced conduction problem~\eqref{diff_problem}. Figures~\ref{heat_4x4init_om},\ref{heat_4x4final_om} show the temperature distribution in a geothermal field composed of $16$ exchangers placed on a regular lattice and in the case of homogenous soil thermal properties; in particular, these figures are relative to the horizontal section of the field taken at half depth, that is $z=20$ m. Figure~\ref{heat_4x4init_om} shows the temperature distribution at the initial time, where the soil has a uniform temperature that corresponds to the value of the undisturbed profile at half depth, i.e. $286$ K, while the exchangers are colder (at $280$ K that is the temperature of the incoming water) and correspond to the blue points. Figure~\ref{heat_4x4final_om} shows the temperature distribution at the final time step, i.e., after $180$ days of operation of the geothermal system; in particular, Figure~\ref{heat_4x4final_om}\subref{heat_4x4final_om_NUM} shows the numerical solution computed by the finite difference scheme, i.e., formula~\eqref{findiff_explicit}, whereas Figure~\ref{heat_4x4final_om}\subref{heat_4x4final_om_AN} shows the analytical solution computed by~\eqref{sol_an_discr}. In both these figures, cold areas arise in neighbourhoods of the exchangers. In particular, Figure~\ref{heat_4x4final_om}\subref{heat_4x4final_om_NUM} shows a more evident interference among exchangers but the areas in the cross points of the diagonals among four devices are almost undisturbed being at $285.9$ K; in Figure~\ref{heat_4x4final_om}\subref{heat_4x4final_om_AN}, we can see a situation similar to Figure~\ref{heat_4x4final_om}\subref{heat_4x4final_om_NUM}, with a slightly slower heat diffusion. The soil affected by interference in Figure~\ref{heat_4x4final_om}\subref{heat_4x4final_om_NUM} differs in temperature from the corresponding areas in Figure~\ref{heat_4x4final_om}\subref{heat_4x4final_om_AN} for at most $0.1$ degrees. Thus, from the qualitative point of view there is a good agreement between the two solutions. For a quantitative comparison of the two solutions, we consider the relative error~\eqref{rel_error} where $w_{i,j,k}^n$ is replaced by $u_{i,j,k}^n$ in~\eqref{sol_an_discr} and $v_{i,j,k}^n$ by $U_{i,j,k}^n$ in~\eqref{findiff_explicit}. The diagram in Figure~\ref{err-heat_4x4_om} shows this relative error for the various simulation time steps; more precisely, we can observe $err^1\approx2.2\cdot10^{-5}$ at the initial time and $err^{N_t}\approx4.2\cdot10^{-4}$ at the end. Different arrangements of the geothermal field show similar results, providing an experimental verification of the reliability of the analytical solution.
\begin{figure}[t!]
\centering
\includegraphics[width=.5\textwidth]{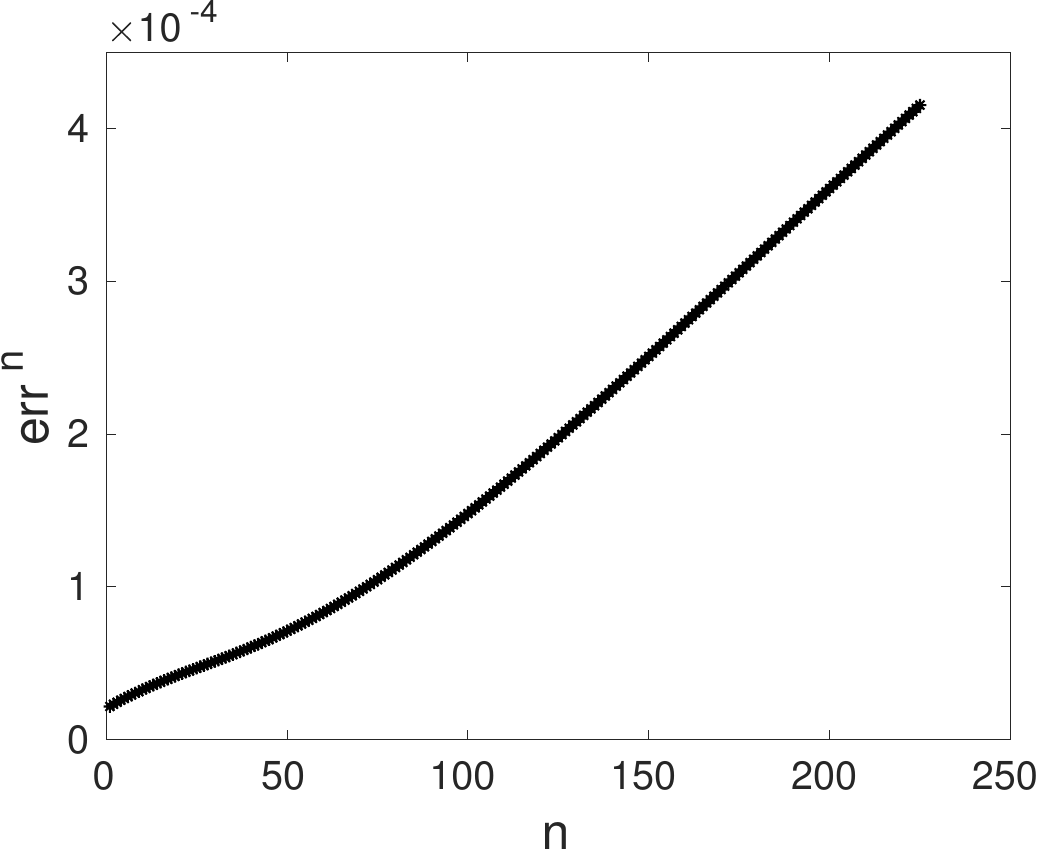}
\caption{Relative error between the finite difference solution and the analytical approach at each time step $n$ of the simulation, in the case of a lattice arrangement with $16$ exchangers.}\label{err-heat_4x4_om}
\end{figure}

\begin{figure}[h!]
\centering
\subfloat[\label{optim_4x4init_om}]{\includegraphics[width=.45\textwidth]{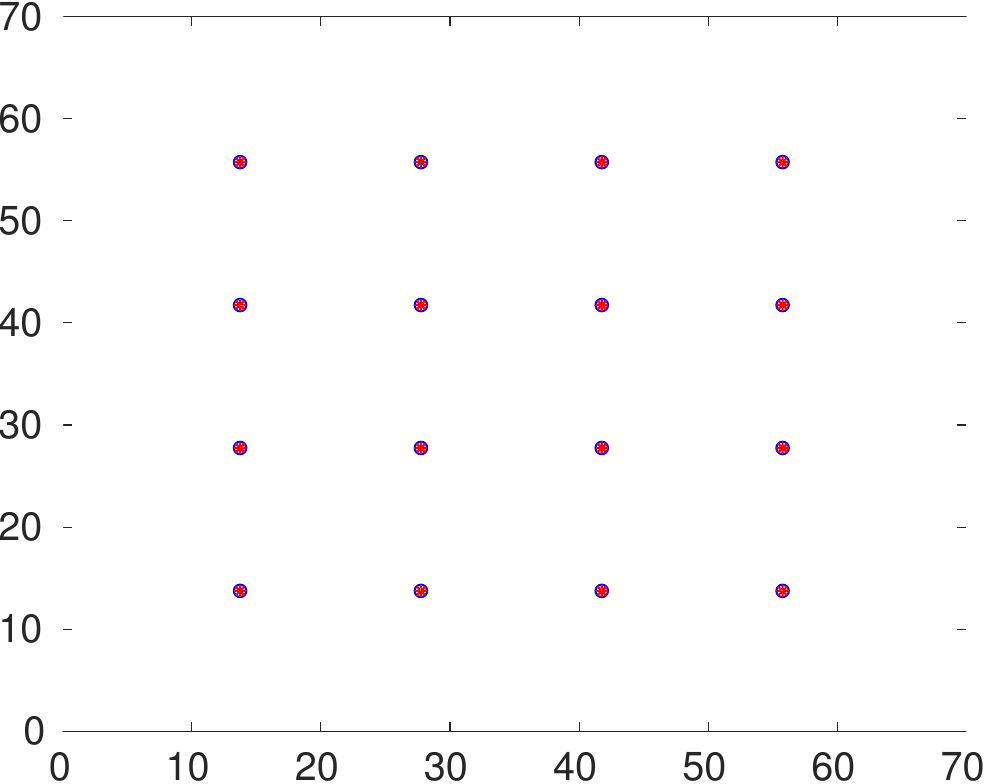}} \hspace{10pt}
\subfloat[\label{optim_4x4final_om}]{\includegraphics[width=.45\textwidth]{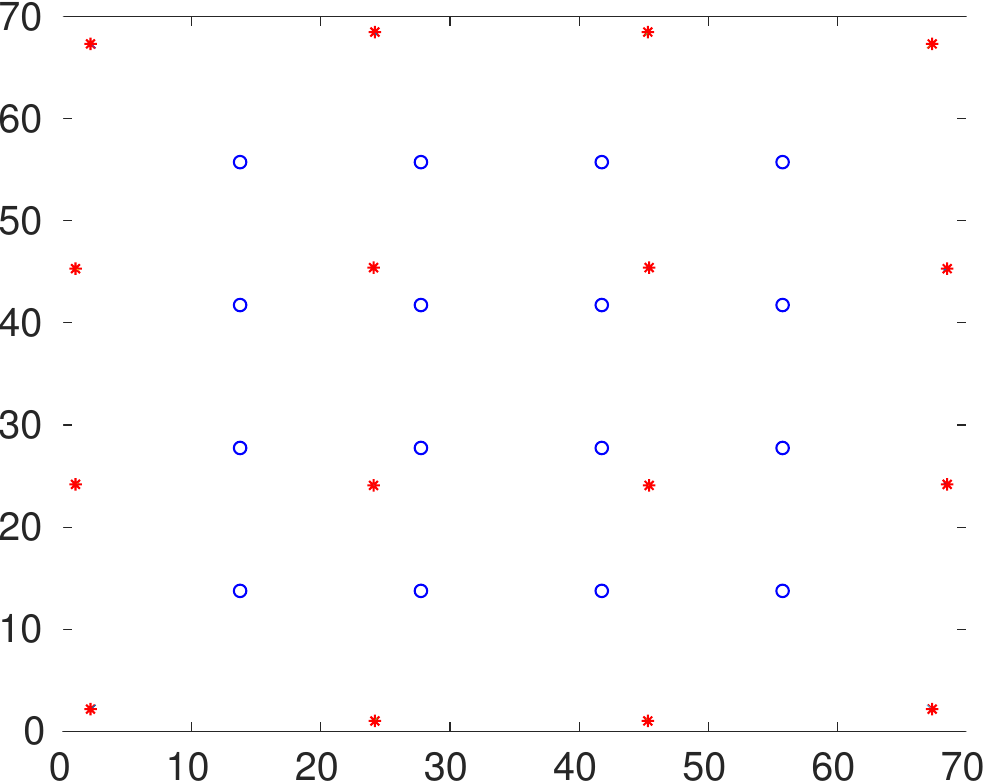}}
\caption{Optimisation method: initial~\protect\subref{optim_4x4init_om} and optimal~\protect\subref{optim_4x4final_om} arrangements of $16$ exchangers.}\label{optim_4x4_om}
\vspace{10pt}
\subfloat[\label{heat-optim_4x4final_om_NUM}]{\adjincludegraphics[width=.50\textwidth,trim={{.025\width} {.008\height} 0 0},clip]{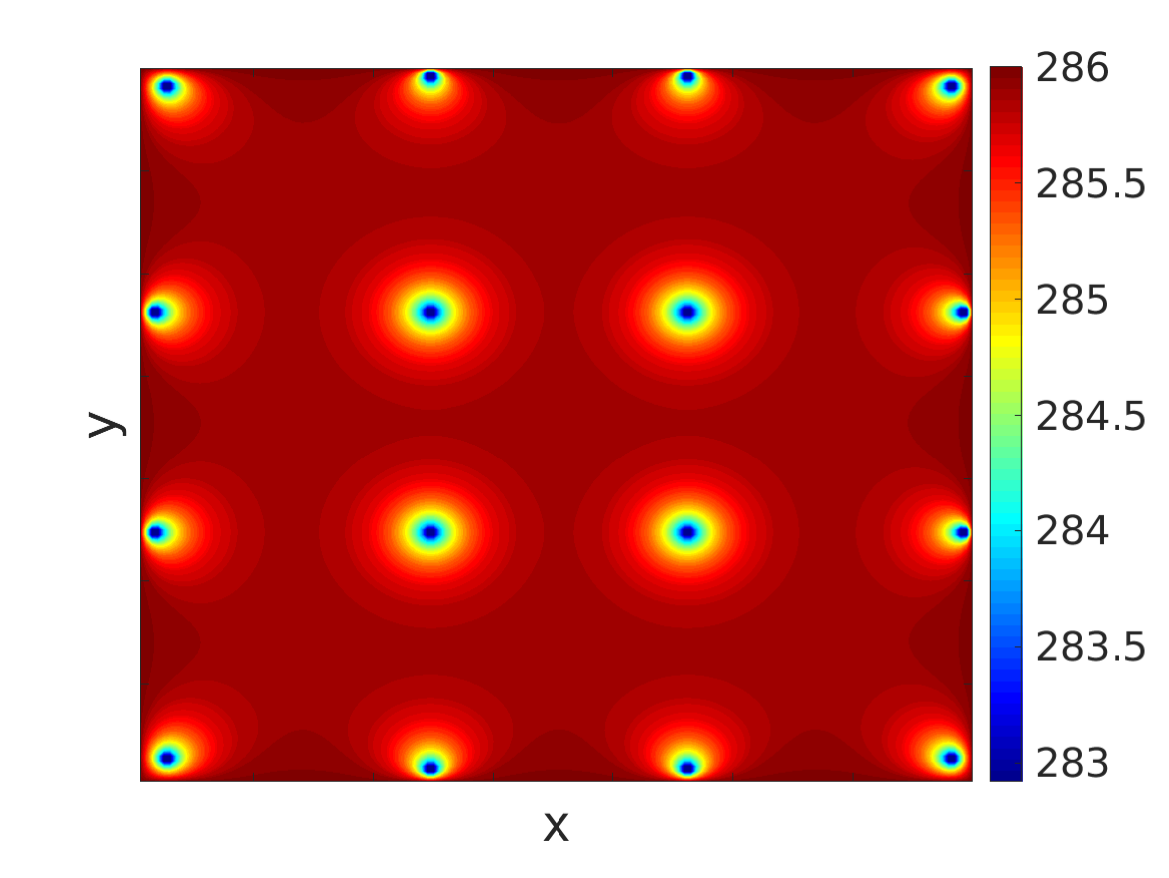}}
\subfloat[\label{heat-optim_4x4final_om_AN}]{\adjincludegraphics[width=.50\textwidth,trim={{.025\width} {.008\height} 0 0},clip]{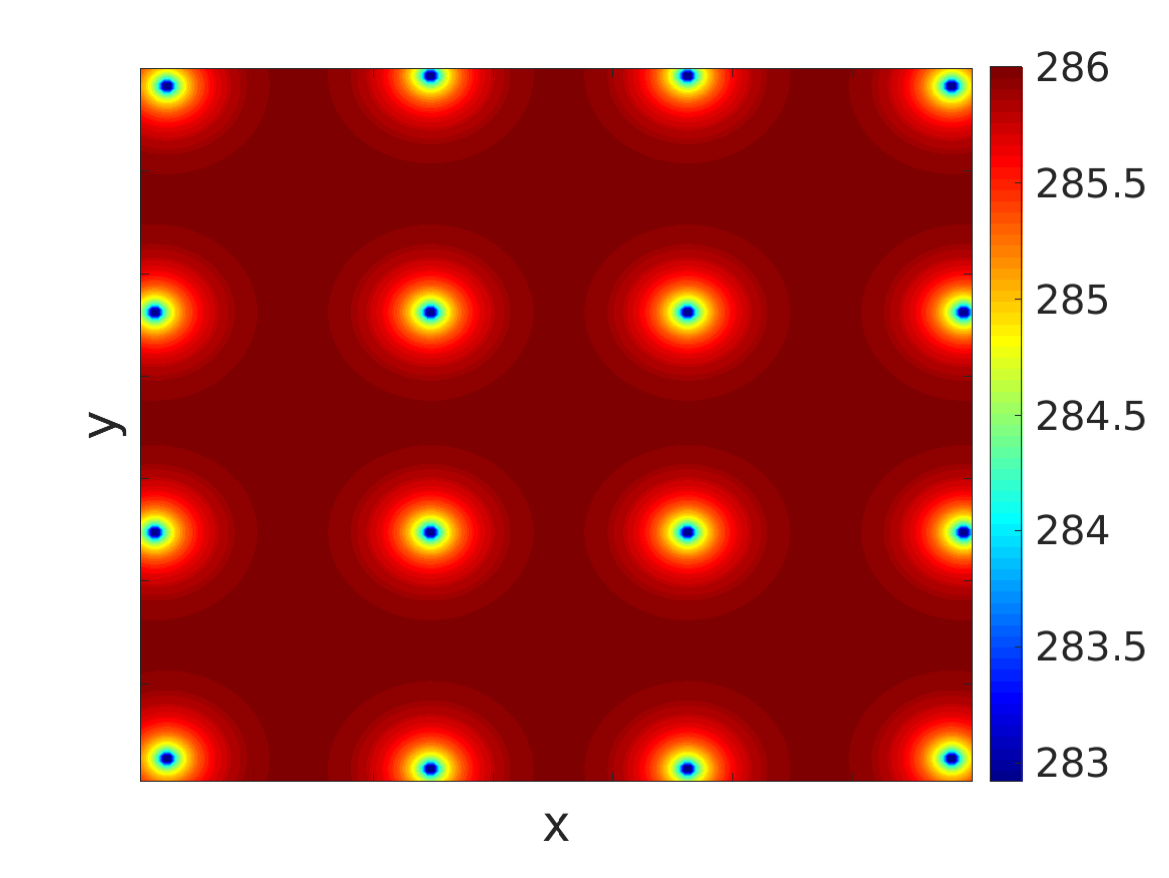}}
\caption{Temperature distribution after $180$ days for the optimal arrangement with $16$ exchangers:~\protect\subref{heat-optim_4x4final_om_NUM} numerical results,~\protect\subref{heat-optim_4x4final_om_AN} analytical results.}\label{heat-optim_4x4final_om}
\end{figure}
\begin{figure}[t]
\centering
\includegraphics[width=.5\textwidth]{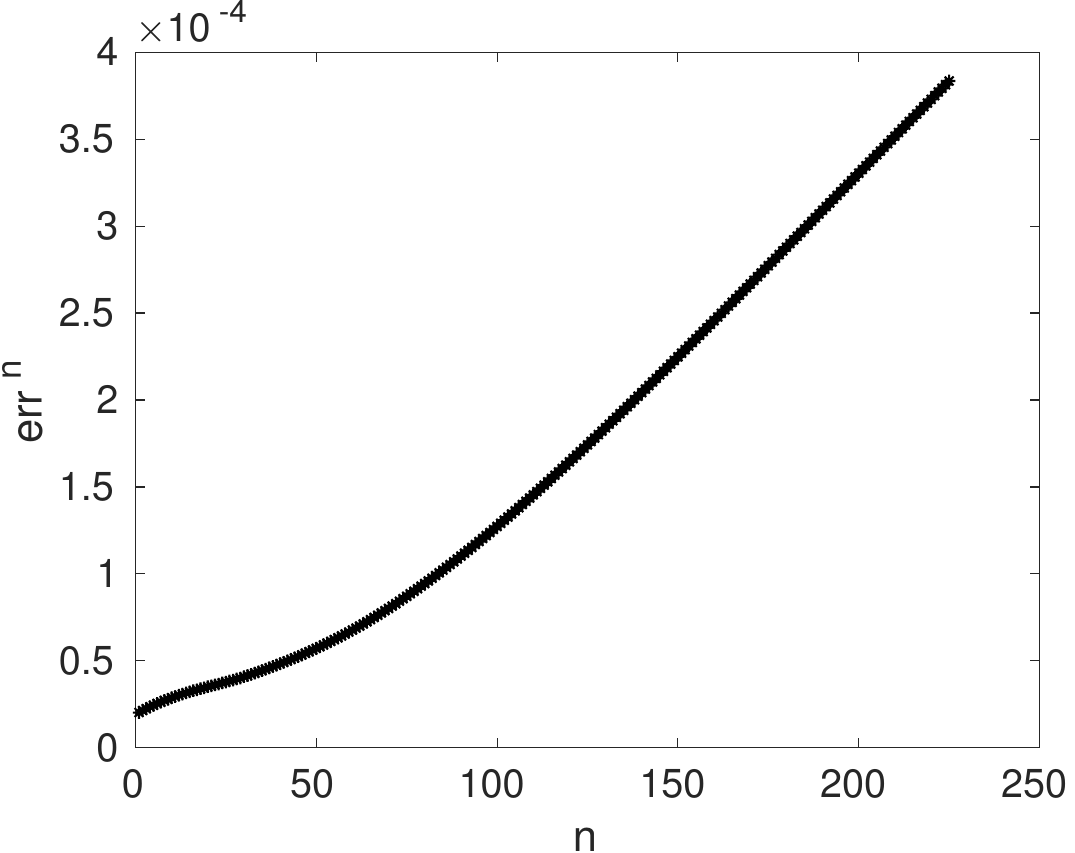}
\caption{Discrete relative error between the numerical and analytical computations at each time step $n$ of the simulation, in the case of $16$ exchangers with arrangement obtained by the optimisation procedure.}\label{err-heat-optim_4x4_om}
\end{figure}
In the next experiment we consider the optimisation procedure described in Section~\ref{sec:optimization}. The uniform lattice, like the one in Figure~\ref{heat_4x4init_om}, is taken as the initial guess for the iterative algorithm. The tolerances and the maximum number of steps in~\eqref{stop_crit} have been chosen as follows: $tol_1=0.1$, $tol_2=2$, $\overline n=150$. The evaluation points in $\mathcal{X}$ are not fixed throughout the process but at each step $n$ they are chosen at regular distance with respect to the exchangers position $\bp^n$. Such strategy prevents unfortunate situations in which an exchanger moves to an evaluation point. Figure~\ref{optim_4x4_om} shows the initial and final arrangements of the exchangers, the little blue circles are the starting positions and the red stars are the optimal positions obtained from problem~\eqref{min_problem}. In Figure~\ref{optim_4x4_om}\subref{optim_4x4final_om}, the devices are moved away from each other but the symmetry of the initial configuration is maintained. The central devices are kept near their original position while the external ones are pushed towards the boundary to better exploit the useful space and allow a bigger distance among all the exchangers. Also the heat exchange between soil and devices in the optimal arrangement has been computed; Figure~\ref{heat-optim_4x4final_om} reports these results. The temperature distributions computed by the finite difference method (Figure~\ref{heat-optim_4x4final_om}\subref{heat-optim_4x4final_om_NUM}) and the analytical approach (Figure~\ref{heat-optim_4x4final_om}\subref{heat-optim_4x4final_om_AN}) are in good agreement. As in the previous experiment, the numerical solution develops a slightly bigger interference among exchangers, but the maximum difference between these two solutions in the soil portions affected by the influence of multiple exchangers is again no bigger than $0.1$ degrees. Comparing the two solutions, an interesting feature is that in Figure~\ref{heat-optim_4x4final_om}\subref{heat-optim_4x4final_om_NUM} the colder areas around the boundary devices lose the circular symmetry and tend to develop mostly in the direction of the neighbouring exchangers, while the mitigating effect of the boundary condition slightly advances among the boundary devices; on the other hand, in Figure~\ref{heat-optim_4x4final_om}\subref{heat-optim_4x4final_om_AN} all the devices develop around colder areas with circular symmetry since there is no boundary condition. Also for this case, the relative error~\eqref{rel_error} observed during the simulation ranges from $2.0\cdot10^{-5}$ to $3.8\cdot10^{-4}$, see Figure~\ref{err-heat-optim_4x4_om}. We note that the relative errors in the case of 16 exchangers with regular arrangement (Figure~\ref{err-heat_4x4_om}) and the case with optimal arrangement (Figure~\ref{err-heat-optim_4x4_om}) are very similar, meaning that the accuracy of the analytical solution is independent of the position of the devices. We also note that the effectiveness of the optimal positioning is clearly visible by the soil temperature distribution. In fact, a brief comparison between Figures~\ref{heat_4x4final_om} and~\ref{heat-optim_4x4final_om} shows that the thermal impact in the soil on the long-term conditions can be effectively reduced by the optimal positioning of the devices, even applying a constant heating load. In Figure~\ref{heat-optim_4x4final_om}, the thermal effect of the exchangers in the soil is mitigated as a consequence of the optimisation process, as can be seen from the bigger undisturbed soil portions among the exchangers.

\begin{figure}[h!]
\centering
\includegraphics[width=.48\textwidth]{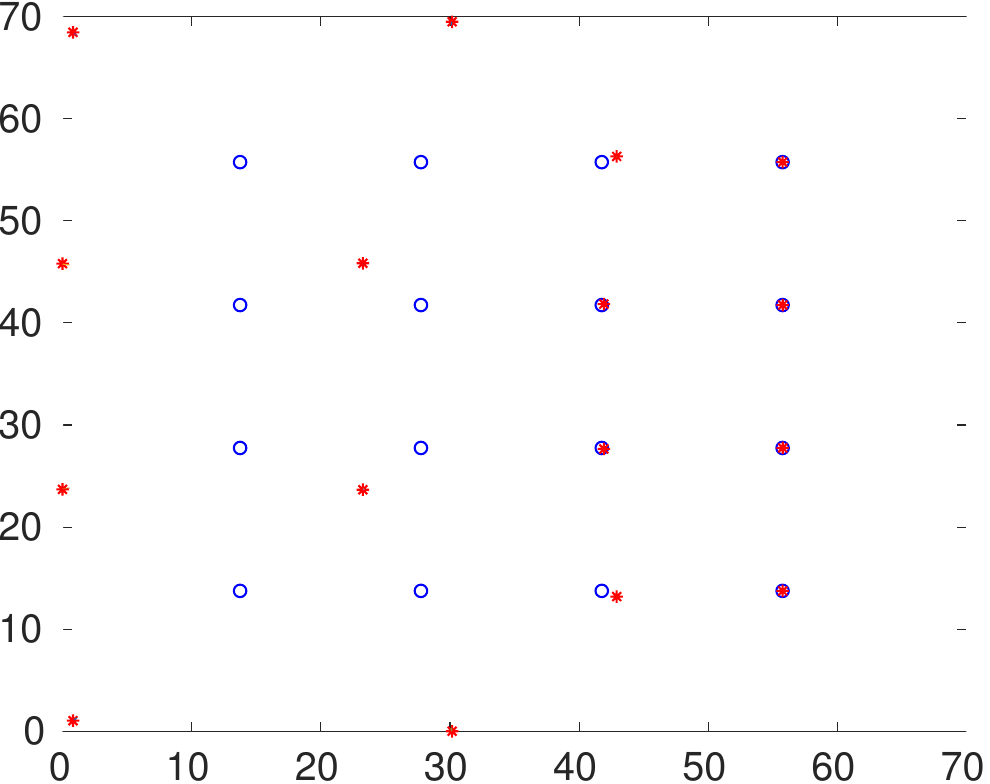}
\caption{Final arrangements of $16$ exchangers in a heterogeneous geothermal field bisected by extremely different thermal diffusivities.}\label{optim_4x4_et}
\vspace{10pt}
\subfloat[\label{optim_5x6init_et}]{\includegraphics[width=.48\textwidth]{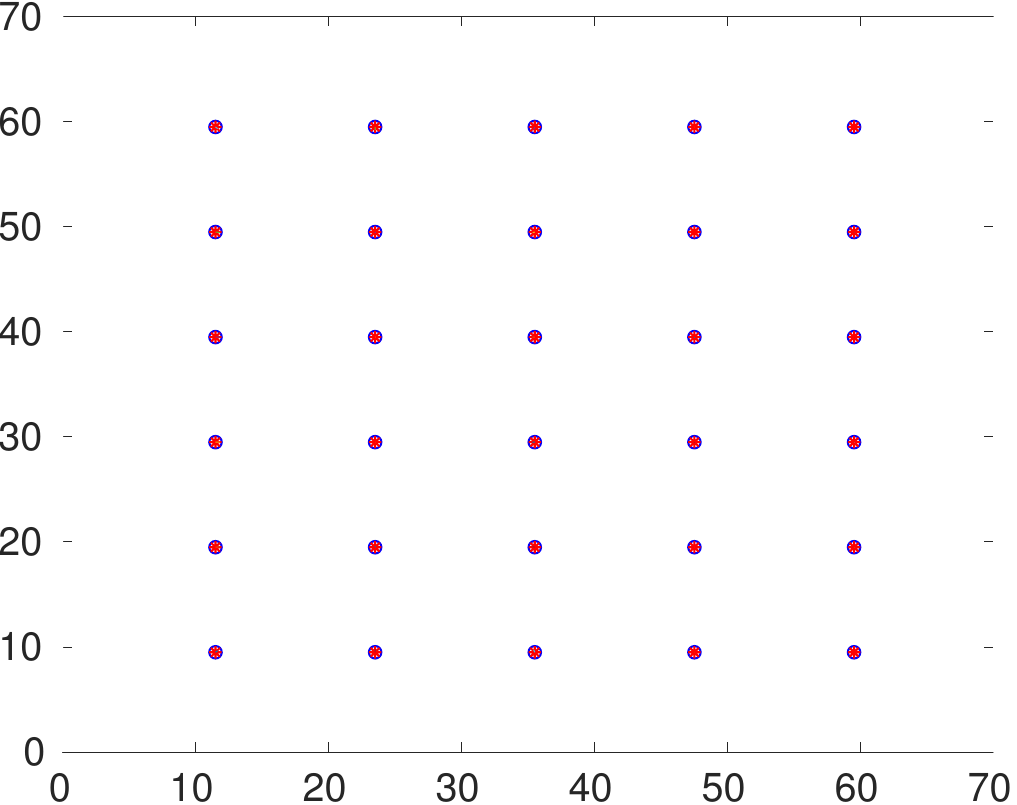}} \hspace{10pt}
\subfloat[\label{optim_5x6final_et}]{\includegraphics[width=.48\textwidth]{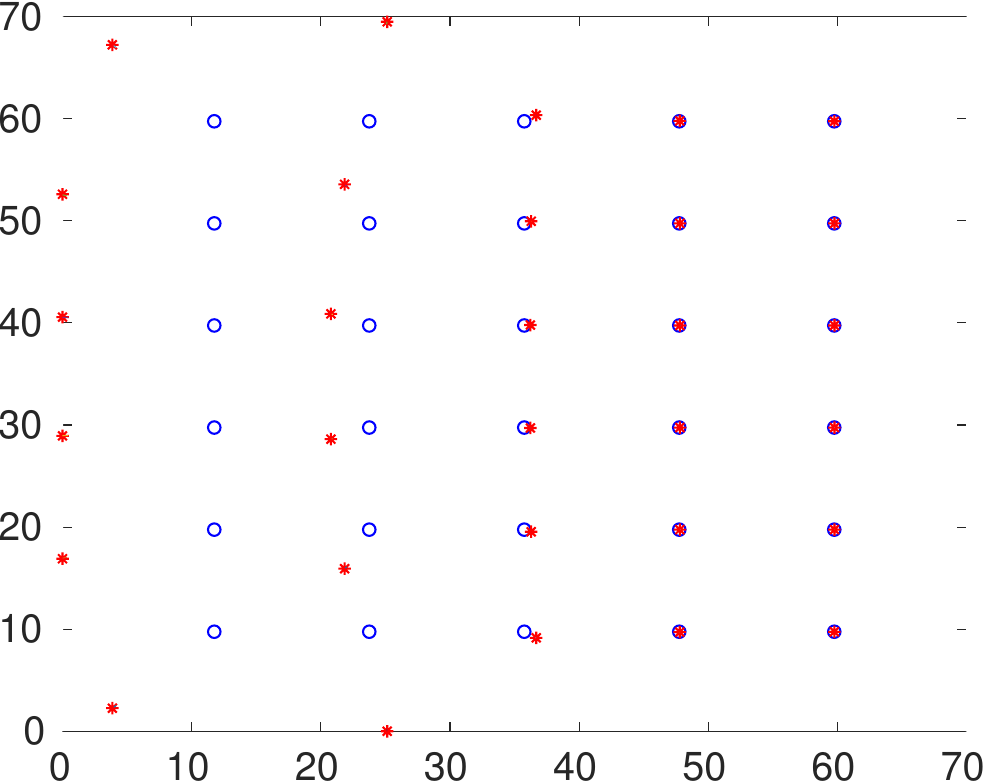}}
\caption{Optimisation method: initial~\protect\subref{optim_5x6init_et} and optimised~\protect\subref{optim_5x6final_et} arrangements of $30$ exchangers in a heterogeneous geothermal field bisected by extremely different thermal diffusivities.}\label{optim_5x6_et}
\end{figure}
The advantage of our analytical approach is to make easier the solution computation of the heat problem, dividing it into the calculation of the direct solution by the Green's function and the source computation by an inverse problem technique, which is rather fast since it is independent of the position of the exchangers on the $x$-$y$ plane. Moreover, this analytical approach makes feasible the proposed optimisation technique. We conclude the numerical experiment with two examples that emphasise the value of the physical-based approach used in the proposed optimisation procedure. Indeed, the previous example can be also solved by using a simple geometric optimisation approach, where the uniform distribution of the devices is computed with respect to the available space. On the contrary, a geometric approach cannot properly deal with a heterogeneous soil situation. So, in the following example, the field is assumed made up of two equivalent parts with extremely different values for thermal diffusion. The two soils are separated by the vertical median of $\hat D$; the soil in the left side has thermal diffusion $0.39$ m$^2$/day, and in the right side $0.022$ m$^2$/day; these values correspond to the mean diffusivity for dry gravel and rock salt~\cite{soilThermalDiff}, respectively. The stopping criteria in problem~\eqref{min_problem} are defined as before. Figure~\ref{optim_4x4_et} shows the final configuration of $16$ exchangers in the heterogeneous geothermal field. The initial configuration is the same of Figure~\ref{optim_4x4_om}\subref{optim_4x4init_om}. In Figure~\ref{optim_4x4_et}, the exchangers initially placed in the left side of the domain tend to maximise the distance among each other exploiting all the surrounding space. The two central devices increase their respective distance moving towards the exterior without going far from their original position. On the contrary, the exchangers initially placed in the right side of the domain keep their original position, except for the first and last exchanger of the third column, which move slightly away due to the influence of the two devices placed at the boundary points with $x$ coordinate equal to $30$. It is clear from Figure~\ref{optim_4x4_et} that in each of the two parts with constant thermal diffusivity, the final configurations maintain the symmetry with respect to the axis $y=35$, in accordance with the symmetry of the problem data; of course, other symmetry properties, like global symmetry, are no longer obtained in the final solution. Figure~\ref{optim_5x6_et} shows the initial and final configurations of $30$ exchangers into the heterogeneous geothermal field. The arrangement reached in Figure~\ref{optim_5x6_et}\subref{optim_5x6final_et} is similar to that in Figure~\ref{optim_4x4_et}; in fact, the optimisation algorithm favours devices placed along the domain boundary over the central positions in the soil with higher diffusivity, and leave the arrangement quite unchanged in the other part. Another noteworthy feature is that the central column of exchangers occurs on the separation line between the two areas with largely different diffusivities, therefore these devices slightly shift to the right, where the thermal exchange is lower. In conclusion, these two tests on a heterogeneous field show that the proposed optimisation process is not a geometrical optimisation procedure, but it actually takes into account the physical characteristics of the problem.

\begin{table}[ht]
\caption{Computation times of the optimisation procedure for the proposed examples.}\label{tab:time_optim_proc}
\centering
\begin{equation*}
{\begin{array}{llcc} \toprule
\multicolumn{2}{c}{\te{Configuration}} & \te{Iterations} & \te{Time [min]} \\ \midrule
\te{homogeneous} & 4\times4 - \te{Fig.}~\ref{optim_4x4_om} & 105 & 32 \\ \midrule
\multirow{2}{*}{heterogeneous} & 4\times4 - \te{Fig.}~\ref{optim_4x4_et} & 96 & 66 \\
& 5\times6 - \te{Fig.}~\ref{optim_5x6_et} & 74 & 75 \\
\bottomrule
\end{array}}
\end{equation*}
\caption{Computation times of one iteration of the optimisation procedure based on the analytical approach and the numerical approach for different space steps, in the case of homogeneous soil with $16$ exchangers.}\label{tab:time_optim_an_num}
\centering
\begin{equation*}
{\begin{array}{lcc} \toprule
\multirow{2}{*}{$h_x(=h_y)$ [m]} & \multicolumn{2}{c}{\te{Time [sec]}} \\
& \te{Analytical method} & \te{Finite difference approach} \\ \midrule
1 & 17 & 20\cdot65=1300\approx 22 \te{ min} \\
0.5 & 17 & 21\cdot65=1365\approx 23 \te{ min} \\
0.25 & 18 & 188\cdot65=12220\approx 3 \te{ h } 24 \te{ min} \\
\bottomrule
\end{array}}
\end{equation*}
\end{table}
Finally, in Table~\ref{tab:time_optim_proc} we show the computation time of the optimisation procedure for the examples considered in this section. All the simulations have been run in the machine HP ProLiant ML350 Gen9, equipped with Red Hat Enterprise Linux Workstation release 7.5 (Maipo) and 2 CPU Intel(R) Xeon(R) CPU E5-2620 v3 @2.40GHz, with 6 physical cores and 6 threads for each CPU. In particular, the simulation code has been written in Matlab and exploits one core. The computation time obviously increases as the number of exchangers increases. The less time-demanding procedure is that one for 16 exchangers in the homogeneous soil, where we have a shorter computation time of each iteration than for the heterogeneous soil case. This fact is mainly due to a simplified algorithm implementation in the homogeneous case. On the contrary, we can observe that the number of iterations for the heterogeneous case study is slightly lower than the one for the homogeneous case study. This is mainly due to the fact that in the homogeneous case the gradients in the steepest descent method are smaller than the ones in the heterogeneous case.

In Table~\ref{tab:time_optim_an_num}, we also show a comparison between the computation time of a single iteration of the proposed optimisation procedure and the one of a similar optimisation procedure where the analytical method is replaced by the finite difference method for the heat diffusion computation. Each step of the proposed optimisation procedure consists in, as described in Section~\ref{sec:optimization}, the calculation of the solution~\eqref{sol_an_discr} and its derivatives with respect to the exchangers position that are contained in formula~\eqref{F_deriv}. On the other hand, each step of the optimisation procedure based on the finite difference approach consists in the calculation of the solution $U$ by using the time-marching scheme~\eqref{findiff_explicit} and the finite difference approximation of its derivatives, that requires $4N_E+1$ evaluation of the solution $U$, thus the computation of $U$ in the whole space and time grid must be repeated $4N_E+1$ times. We note that in this way we simply have a lower estimate of the computation time for the finite difference approach. In particular, Table~\ref{tab:time_optim_an_num} considers the configuration with homogeneous soil and $16$ exchangers, operating for $120$ days, where the computation is performed by using three grids with increasing fineness having smaller and smaller step size $h_x=h_y$ along the $x,y$ directions and a constant step size $h_z=0.5$ m along the $z$ direction. Even the coarsest discretisation shows a big gap in the elapsed time for the two computational approaches, where the analytical method is about $76$ times faster than the numerical approach. The reduction rate is increasingly higher as the discretisation steps reduce and for the finest grid the proposed optimisation procedure is about $680$ times faster than the procedure based on the finite difference approach. We remind that such a gap in the computation time is also due to the fact that a reduction of the space step induces a reduction of the time step, according to the Courant-Friedrichs-Lewy condition. Instead, in the analytical method the usual time step $\Delta t=0.8$ days has been used, and it could be even increased to make the computation faster even keeping a comparable accuracy. In addition, the optimisation procedure based on the finite difference scheme reaches the optimal solution only if the space grid is sufficiently fine, that is $h_x,h_y\leq0.25$ m. Hence, Table~\ref{tab:time_optim_an_num} shows that the finite difference approach used within an optimisation procedure is definitely unfavourable in terms of computation cost.

\section{Conclusions}\label{sec:conclusions}

The paper studies the geothermal fields of borehole heat exchangers for the exploitation of shallow geothermal energy resource. This study considers the coupling of the soil and exchangers thermal response, which consists in taking into account the mutual influence between the cooling of the soil and the corresponding warming of the exchanger (in winter operational mode). Two main contributions are obtained: the development of an analytical method to quantify the temperature evolution within the geothermal field, and the development of an optimisation procedure for the positions of the devices. These two results are strictly related, in fact, the heat exchange in the geothermal field is given in terms of the fundamental solution of the heat equation, allowing a direct connection between the exchangers position and the temperature field in the soil. The soil temperature obtained with the proposed method is compared with the result obtained by standard approximation methods, i.e., the Green's function approach in the case of a unique exchanger, and  the finite difference method in the case of multiple exchangers. From this numerical experiment, we can conclude that the proposed method is reliable for describing the temperature distribution in a geothermal field. Moreover, the optimisation procedure is tested by using some reference cases, showing again a good reliability of the proposed procedure. It is also shown that the induced thermal anomalies are reduced by the optimisation process.

Despite the good results obtained with the proposed method, some assumptions made in the paper could be further refined. For instance, the undisturbed soil thermal profile could be adapted to be as close as possible to a real profile. In addition, the inverse problem for the estimation of the source term that mimics the exchangers presence should take into account the interference among sufficiently near devices. Finally, on-field experiments for the temperature measurements in a real geothermal field should be performed to obtain a practical validation of the proposed approach.



\printbibliography


%
%
%
\end{document}